\documentclass{amsart}
\usepackage[utf8]{inputenc}
\usepackage{a4wide}
\usepackage{amsmath,amsfonts,amssymb}
\usepackage{amsthm}
\usepackage{ifthen}
\usepackage{longtable}
\usepackage{array}
\usepackage{url}
\usepackage{listings}
\usepackage{graphicx}
\usepackage{mathtools}
\usepackage[shortlabels]{enumitem}
\usepackage{tabularray}

\usepackage{xcolor}
\usepackage{hyperref}
\usepackage{nameref}

\makeatletter
\newcommand\setcurrentname[1]{\def\@currentlabelname{#1}}
\makeatother

\hypersetup{colorlinks,
	linkcolor=green!50!black,
	citecolor=blue!70!black,
	urlcolor=red!50!black
}

\newcommand{\Z}{\mathbb{Z}}
\newcommand{\Q}{\mathbb{Q}}
\newcommand{\R}{\mathbb{R}}
\newcommand{\C}{\mathbb{C}}

\newcommand{\N}{\mathbb{N}}

\newcommand{\eps}{\varepsilon}
\newcommand{\bb}{,\ldots ,}

\DeclarePairedDelimiter\floor{\lfloor}{\rfloor}

\newcommand{\hf}{{}_2\mathcal{F}_1}
\newcommand{\qi}{\Z_{\Q(\sqrt{-d})}}
\newcommand{\Qd}{\Q(\sqrt{-d})}
\newcommand{\pA}{\mathbf{A}}
\newcommand{\pB}{\mathbf{B}}
\newcommand{\pC}{\mathbf{C}}
\newcommand{\pP}{\mathbf{P}}
\newcommand{\pU}{\mathbf{U}}
\newcommand{\pa}{\mathbf{a}}
\newcommand{\pb}{\mathbf{b}}
\newcommand{\pc}{\mathbf{c}}
\newcommand{\pd}{\mathbf{d}}
\newcommand{\pu}{\mathbf{u}}
\newcommand{\pz}{\mathbf{z}}
\newcommand{\pw}{\mathbf{w}}
\newcommand{\pY}{\mathbf{Y}}

\DeclareMathOperator{\disc}{disc}
\DeclareMathOperator{\id}{id}

\newcommand{\case}[1]{\noindent\emph{Case #1:}}

\newtheorem*{theorem*}{Theorem}
\newtheorem{thm}{Theorem}
\newtheorem{lem}{Lemma}
\newtheorem{cor}{Corollary}
\newtheorem{prop}{Proposition}

\newtheorem*{claim*}{Claim}

	\newcounter{countknownthm}
	
\newtheorem{knownthm}[countknownthm]{Theorem}
	\newcounter{countknownlem}
	
\newtheorem{knownlem}[countknownlem]{Lemma}

	\newtheorem{innercustomlem}{Lemma}
\newenvironment{customlem}[1]
  {\renewcommand\theinnercustomlem{#1}\innercustomlem}
  {\endinnercustomlem}

\theoremstyle{definition}

	\newtheorem{rem}{Remark}
	
\newtheoremstyle{stepstyle}{0pt}{\topsep}{\normalfont}{0pt}{\itshape}{:}{5pt plus 1pt}{}

\theoremstyle{stepstyle}
\newtheorem{innerstep}{Step}
\newtheorem{innersteps}{Steps}

\newenvironment{step}[1]
  {\renewcommand\theinnerstep{{\itshape #1}}\innerstep}
  {\endinnerstep}
\newenvironment{steps}[1]
  {\renewcommand\theinnersteps{{\itshape #1}}\innersteps}
  {\endinnersteps}	
	
\usepackage{titlesec}
\titleformat{\section}
  {\center\normalsize\bfseries}{\thesection.}{1em}{}
\titleformat{\subsection}
  {\normalsize\bfseries}{\thesubsection.}{1ex}{}

\begin{document}
\title[Family of Thue equations over imaginary quadratic fields]{On a simple quartic family of Thue equations over imaginary quadratic number fields}
%11N25 Distribution of integers with specified multiplica-tive constraints
%11D61  	Exponential Diophantine equations
%  11B39	 	Fibonacci and Lucas numbers and polynomials and generalizations
% 11B37   	Recurrences {For applications to special functions, see 33-XX}
%  	11D45   	Counting solutions of Diophantine equations
% 11D75   	Diophantine inequalities [See also 11J25]
%\thanks{I.V. was supported by the Austrian Science Fund (FWF) under the project I4406.
%}

\author[B. Earp-Lynch]{Benjamin Earp-Lynch}
\address{B. Earp-Lynch,
	Carleton University,
    1125 Colonel By Dr, Ottawa, ON, Canada}
\email{benjaminearplynch@cmail.carleton.ca}

\author[B. Faye]{Bernadette Faye}
\address{B. Faye,
	UFR SATIC, Universit\'e Alioune Diop de Bambey,
	Diourbel, Bambey 30, S\'en\'egal}
\email{bernadette.faye@uadb.edu.sn}

\author[E. G. Goedhart]{Eva G. Goedhart}
\address{E. Goedhart,
Franklin \& Marshall College,
415 Harrisburg Ave,
Lancaster, PA 17603, United States of America}
\email{eva.goedhart\char'100fandm.edu}

\author[I. Vukusic]{Ingrid Vukusic}
\address{I. Vukusic,
University of Salzburg,
Hellbrunnerstrasse 34/I,
5020 Salzburg, Austria}
\email{ingrid.vukusic\char'100plus.ac.at}

\author[D. P. Wisniewski]{Daniel~P.\ Wisniewski}
\address{D. P. Wisniewski,
Department of Mathematics/Computer Science,
DeSales University,
Center Valley, PA 18034, United States of America
}
\email{daniel.wisniewski@desales.edu}

%\author[V. Ziegler]{Volker Ziegler}
%\address{V. Ziegler,
%University of Salzburg,
%Hellbrunnerstrasse 34/I,
%5020 Salzburg, Austria}
%\email{volker.ziegler\char'100plus.ac.at}

\begin{abstract}
Let $t$ be any imaginary quadratic integer with $|t|\geq 100$. We prove that the inequality
\[
	|F_t(X,Y)|
	 = |  X^4 - t X^3 Y - 6 X^2 Y^2  + t X Y^3 + Y^4 |
	\leq 1
\]
has only trivial solutions $(x,y)$ in integers of the same imaginary quadratic number field as $t$. 
Moreover, we prove results on the inequalities $|F_t(X,Y)| \leq C|t|$ and $|F_t(X,Y)| \leq |t|^{2 -\eps}$.
These results follow from an approximation result that is based on the hypergeometric method. The proofs in this paper require a fair amount of computations, for which the code (in Sage) is provided.
\end{abstract}

\maketitle

\section{Introduction}

In $1909$, Thue \cite{Thue1909} proved that if $F(X,Y) \in \Z[x,y]$ is an irreducible form of degree at least $3$, and $m$ is a nonzero integer, then the Diophantine equation (called a {\it{Thue equation}})
$$F(x,y)=m$$
has only finitely many solutions over the integers. With the development of the theory of lower bounds for linear forms in logarithms and reduction methods, it became possible to solve specific Thue equations completely.
Since the late 1980s, there exist algorithms to solve single Thue equations completely (see in particular the method of Tzanakis and De Weger \cite{TzanakisDeWeger1989}). 

In 1990, Thomas \cite{Thomas1990} considered the parametrized family of Thue equations
\begin{equation}\label{eq:Thomas}
	F^{(3)}_t(X,Y)		
	:= X^3 - (t-1)X^2 Y - (t+2) X Y^2 - Y^3  
	= 1
\end{equation}
over integers. The resolution of \eqref{eq:Thomas} was completed shortly afterwards by Mignotte \cite{Mignotte1993}.

Since then, various families of Thue equations have been solved (see \cite{Heuberger2006-surv} for a survey from 2005). 

In particular, we want to mention that the form in equation~\eqref{eq:Thomas} is called a ``simplest'' cubic form, as defined by Lettl et al. \cite{LettlPethoVoutier1999}. In their paper, Lettl et al. moreover considered the higher degree ``simplest forms''
\begin{align*}
	F^{(4)}_t(X,Y)		
	&:= X^4 - t X^3 Y - 6 X^2 Y^2  + t X Y^3 + Y^4, \\
	%\quad \text{and} \\
	F^{(6)}_t(X,Y)
	&:= X^6 - 2 t X^5 Y - (5t + 15) X^4 X^2 - 20 X^3 Y^3 + 5 t X^2 Y^4 + (2t + 6) X Y^5  + Y^6.
\end{align*}
They solve inequalities of the shape 
\[
	| F^{(i)}_t(X,Y)| \leq k(t),
\]
where $k\colon \Z \to \N$ is a function in $t$, for example a linear one. Note that the fields associated with the corresponding univariate polynomials $f_t^{(i)}(X):= F_t^{(i)}(X,1)$ for $i=3,4,6$ were already traditionally called ``simplest fields.'' Lettl et al. \cite{LettlPethoVoutier1999} introduced a formal definition for simplest forms. However, their definition in fact includes more forms than they considered. The full set of simplest forms can be obtained by introducing an extra parameter in $F^{(3)}_t(X,Y)$, $F^{(4)}_t(X,Y)$ and $F^{(6)}_t(X,Y)$. This was pointed out by Wakabayashi~\cite{Wakabayashi2007}, who then also considered those Thue inequalities.

In general, most families of Thue equations and inequalities have been considered in an integer setting, i.e.\ where the parameter(s) and solutions are all integers. However, there also exist results in number field settings (relative Thue equations) and in the function field setting. 
In the function field setting, many different families have been studied; let us just mention that simplest families have been studied in \cite{FuchsZiegler2006} and \cite{FayeVukusicWaxmanZiegler2023}.
We now focus on relative Thue equations, i.e.\ Thue equations, where the coefficients and solutions are in a fixed number field. The name ``relative'' comes from the fact that if viewed as a norm form equation, the relative norm is taken.

As for classical (absolute) Thue equations, there exist algorithms for solving single relative Thue equations (in particular see \cite{GaalPohst2002}). However, only a few families of relative Thue equations have been studied. To the authors' best knowledge, only the following results exist so far: Families of relative Thue equations and inequalities related to $F^{(3)}_t(X,Y)$ have been solved in \cite{HeubergerPethoTichy2002}, \cite{Heuberger2006} and \cite{KirschenhoferLamplThuswaldner2007}; some ``non-simple'' families of degree 3 and 4 have been studied in \cite{Ziegler2005},  \cite{Ziegler2006} and \cite{JadrijevicZiegler2006}; and recently, Gaál et al. \cite{GaalJadrijevicRemete2019} considered the inequalities $|F^{(4)}_t(X,Y)|\leq 1$ and $|F^{(6)}_t(X,Y)|\leq 1$. All these relative Thue equations and inequalities have been studied over imaginary quadratic number fields, because those are the only number fields where integers stay away from each other and methods from the classical setting can be adapted. In particular, Heuberger \cite{Heuberger2006} completely solved the family of Thue equations
\[
	F^{(3)}_t(X,Y) = \mu,
\]
%where $t$ is an imaginary quadratic integer, $\mu$ is a root of unity in $\Q(t)$, and $x,y\in \Z_{\Q(t)}$.
where the parameter $t$, the root of unity $\mu$ and the solutions $x$ and $y$ are integers in the same imaginary quadratic number field. 
Gaál et al. \cite{GaalJadrijevicRemete2019} completely solved the families $|F^{(4)}_t(X,Y)|\leq 1$ and $|F^{(6)}_t(X,Y)|\leq 1$, but only for rational integer parameters $t$. Their degree 4 result in particular implies the following (we exclude small integer values of $t$ here for simplicity, as some of them lead to sporadic solutions):

\begin{knownthm}\label{thm:Gaal}
Let $t\in \Z$ with $|t| \geq 5$ and let $d$ be a positive square-free integer.
Then the inequality
\begin{equation}\label{eq:Gaal}
	|F_t^{(4)}(X,Y)| \leq 1 
	\quad \text{in } X,Y \in \qi
\end{equation}
has only trivial solutions, i.e.\ solutions of the shape $(0,0)$, $(\xi, 0)$ or $(0,\xi)$, where $\xi$ is a root of unity in $\qi$.
\end{knownthm}

The proof of Theorem~\ref{thm:Gaal} in \cite{GaalJadrijevicRemete2019} is based on a previous paper by the same authors \cite{GaalJadrijevicRemete2018}. There they give a method for reducing the resolution of a relative Thue inequality to the resolution of the corresponding absolute Thue inequality. With this method, they are able to prove Theorem~\ref{thm:Gaal} rather quickly. However, the method only works for forms with integers coefficients. 

In this paper, we want to extend Theorem~\ref{thm:Gaal} and allow the parameter $t$ to be an imaginary quadratic integer as well.
As in \cite{LettlPethoVoutier1999} and \cite{Heuberger2006}, our proof relies on the hypergeometric method. We will solve inequality~\eqref{eq:Gaal} for imaginary quadratic integers $t$ with $|t|\geq 100$ (see Theorem~\ref{thm:main100}).

Moreover, we will prove results on some inequalities with larger upper bound than $1$ (in the style of \cite{LettlPethoVoutier1999}), see Corollaries~\ref{cor:lin} and \ref{cor:eps}. All these results will be based on the approximation result (Proposition~\ref{prop:approx}) obtained from the hypergeometric method. 

The results are presented in the next section, as well as an outline of the rest of the paper (see Table~\ref{tab:overview}).

Finally, let us point out that our goal is to present the proofs 
in full detail and also provide the used Sage code. The code is linked at the appropriate places throughout the paper and the URLs can also be found in the \nameref{sec:appendix}.

\section{Results and outline of the paper}\label{sec:results}

Let us set 
\[
	F_t(X,Y)
	:= F_t^{(4)}(X,Y)
	= X^4 -t X^3 Y - 6 X^2 Y^2 + t X Y^3 + Y^4
\]
and let $d$ be a positive square-free integer. 
We want to investigate the inequality
\begin{equation}\label{ineq:maintemp}
	|F_t(X,Y)|
	\leq 1, 
	\quad \text{in } X,Y \in \qi,
\end{equation}
where $t\in \qi$ and $\qi$ denotes the ring of integers of the number field $\Qd$. 

Before we state our main result, let us split inequality ~\eqref{ineq:maintemp} into equations and discuss some obvious solutions. 

First, note that the absolute value of an imaginary quadratic integer is either 0 (if the integer is 0), 1 (if it is a root of unity) or larger than 1 (in all other cases). Therefore, solving \eqref{ineq:maintemp} is equivalent to solving the two equations $F_t(X,Y)=0$ and $F_t(X,Y)=\mu$, where $\mu$ is a root of unity in $\qi$. The first equation will be solved by arguing that $F_t(X,1)$ is irreducible. The second equation will need more attention.

Next, note that
\[
	F_t(X,Y)
	= F_t(-X,-Y)
	= F_t(-Y,X)
	= F_t(Y,-X)
	%= F_{-t}(Y,X)
	.
\]
Therefore, with every solution $(x,y)$ of \eqref{ineq:maintemp} there usually come three more solutions $(-x,-y)$, $(-y,x)$, $(y,-x)$ and we say that these solutions are \emph{equivalent}. 

Finally, note that if $\xi$ is a root of unity in $\qi$, then $(\xi,0)$ is a solution to $F_t(X,Y)=\mu$ for $\mu = \xi^4$ and arbitrary $t$. We call such a solution, as well as all equivalent solutions, and the solution $(0,0)$ \emph{trivial solutions}.

\begin{rem}\label{rem:trivialSols}
The only roots of unity in imaginary quadratic fields are $\pm 1$, $\pm i$ and $\pm \zeta_6$, $\pm \zeta_6^2$, where $\zeta_6=(1+i\sqrt{3})/2$ is the primitive sixth root of unity. 
Therefore, if $d\notin\{1, 3\}$, we have no trivial solutions for $\mu=-1$ and we have, up to equivalence, one trivial solution for $\mu =1$ and arbitrary $t$, namely $(1,0)$. 
If $d=1$, we have no trivial solutions for $\mu \in \{-1, \pm i\}$ and we have, up to equivalence, two trivial solutions for $\mu =1$, namely $(1,0)$ and $(i,0)$.
If $d=3$, we have no trivial solutions for $\mu \in \{-1, \zeta_6, - \zeta_6^2\}$, we have one trivial solutions for $\mu =1$, one for $\mu = -\zeta_6$ and we have one trivial solution for $\mu = \zeta_6^2$, all up to equivalence and for arbitrary $t$. Those last three solutions are $(1,0)$, $(\zeta_6,0)$ and $(\zeta_6^2,0)$.
\end{rem}

We will prove the following main result.

\begin{thm}\label{thm:main100}
Let $d$ be a positive square-free integer and $t\in \qi$ with $|t|\geq 100$.
Then any solution $(x,y)\in \qi$ to the inequality
\begin{equation}\label{ineq:main}
	|F_t(X,Y)|
	\leq 1.
\end{equation} 
is trivial, i.e.\ of the shape $(0,0)$ or $(\xi, 0)$ or $(0,\xi)$, where $\xi$ is a root of unity.
\end{thm}

The typical strategy for solving Thue equations of the shape $F(X,Y) = c$ (whether with Baker's method or the hypergeometric method) uses the following fact: If $(x,y)$ is a solution to the Thue equation, then $x/y$ is a particularly good approximation to one of the roots of the corresponding univariate polynomial $f(X)=F(X,1)$.
We say that a solution is of type $j$, if it approximates the $j$-th root.

In our case, let us set
\[
	f_t(X)
	:= f_t(X,1)
	= X^4 -t X^3 - 6 X^2 + t X + 1
\]
and let $\alpha$ be a root of $f_t$. Then one can prove (in fact, this is how the simplest quartic forms were constructed in the first place) that the full set of roots of $f_t$ is given by
\[
	\alpha^{(0)} = \alpha, \quad
	\alpha^{(1)} = \frac{\alpha-1}{\alpha+1},\quad 
	\alpha^{(2)} = -\frac{1}{\alpha}, \quad
	\alpha^{(3)} = -\frac{\alpha+1}{\alpha -1}.
\]
This is further elaborated in the proof of Lemma~\ref{lem:irred}. Moreover, as we will see in Lemma~\ref{lem:roots-approx}, one of the roots is close to zero and we will set $\alpha^{(0)}$ to be that root. 
Proposition~\ref{prop:approx} gives us a strong bound for how good general approximations to $\alpha^{(0)}$ and $\alpha^{(3)}$ can get. The proof of Theorem~\ref{thm:main100} will mostly rely on the proposition. 
Note that we only consider approximations to two of the four roots. This is because we will prove in Lemma~\ref{lem:typewlog} that solutions of the other two types are equivalent to solutions of type $0$ or $3$.

\begin{prop}\label{prop:approx}
	Let $t \in \qi$ with $|t|\geq 100$. Let $\alpha^{(0)}=\alpha$ and $\alpha^{(3)}$ be two roots of $f_t(X)$ such that $\alpha$ is the unique root with $|\alpha|\leq 1/4$ and $\alpha^{(3)}=-\frac{\alpha + 1}{\alpha - 1}$. 
Then for any $p,q\in \qi$ with $|q|\geq 0.28|t|$ we have 
\begin{align*}
		\left| \alpha^{(j)} - \frac{p}{q} \right|
	> \frac{1}{15.48 |t| |q|^{\kappa + 1}},
	\quad \text{with} \quad 
	\kappa = \frac{\log |t| + 1.08}
				  {\log |t| - 2.59}
\end{align*}
for $j \in \{0, 3 \}$.
In particular, since $|t|\geq 100$, we have  
\[
	\kappa < 2.83.
\]
\end{prop}

The main strategy for proving Theorem~\ref{thm:main100} will be the following: We compare the lower bound from Proposition~\ref{prop:approx} (which is of the shape $c/(|t| \cdot |y|^{\kappa + 1})$) with an upper bound, which will be of the shape $c/(|t| \cdot |y|^4)$. In other words, we will end up with an inequality of the shape $1/(|t| \cdot |y|^{\kappa + 1}) \ll 1/(|t| \cdot |y|^4)$. This inequality implies an absolute upper bound on $|y|$ as soon as $\kappa<3$.

Now note that $\kappa$ can actually get arbitrarily close to $1$, which means that we are ``wasting'' powers of $|y|$ in our proof.
The next two Corollaries better show the actual power of Proposition~\ref{prop:approx}. We call them Corollaries, because in contrast to Theorem~\ref{thm:main100} they follow relatively quickly from Proposition~\ref{prop:approx}. This is because we exclude small solutions in the statements, whereas for Theorem~\ref{thm:main100} we have to prove lower bounds for $|y|$. The authors don't know how one could prove such general lower bounds in the context of the Corollaries.

\begin{cor}\label{cor:lin}
Let $C>0$ be given. Then there exist effectively computable constants $C_0 >0$ and $t_0\geq 100$, both depending on $C$, such that the following statement holds:
For any square-free integer $d$ and any $t\in \qi$ with $|t|\geq t_0$ the inequality 
\begin{equation*}\label{ineq:cor-lin}
	|F_t(X,Y)|\leq C |t|
	\quad \text{in } X,Y \in \qi
\end{equation*}
has no solutions $(x,y)$ with $\min\{|x|,|y|\}\geq C_0$, except solutions of the shape $(x,\pm x)$ with $|x| \leq (C|t|/4)^{1/4}$.
\end{cor}

\begin{cor}\label{cor:eps}
Let $0 < \eps < 1$ be given. Then there exists an effectively computable constant $t_0\geq 100$ depending on $\eps$, such that the following statement holds:
For any square-free integer $d$ and any $t\in \qi$ with $|t|\geq t_0$ the inequality 
\begin{equation*}
	|F_t(X,Y)|\leq |t|^{2-\eps}
	\quad \text{in } X,Y \in \qi
\end{equation*}
has no solutions $(x,y)$ with $\min\{|x|,|y|\} > (|t|^{2-\eps}/4)^{1/4}$.
%except solutions of the shape $(x,\pm x)$ with $|x| \leq (|t|^{2-\eps}/4)^{1/4}$.
\end{cor}

In Table~\ref{tab:overview} we give an overview of the rest of the paper. Note that we focus on proving Theorem~\ref{thm:main100} throughout the paper. Only in the very last section, we generalize some of the previous results and prove the two Corollaries.

\begin{table}[h]
\caption{Overview of paper}\label{tab:overview}
\begin{tblr}{cl}
\hline
Section                          & goal \\ \hline
	\ref{sec:irred} 
	& Determine all $t$'s for which $f_t(X)$ is reducible; solve $F_t(X,Y)=0$. 
	\\
	\ref{sec:smallsols}
	& Find all solutions to $F_t(X,Y)=\mu$ with $\min\{|x|,|y|\} < 3$.                                                                     	\\
	\ref{sec:approx}
	& Find approximations to the roots of $f_t(X)$, e.g.\ $\alpha = -1/t + L(5.01 |t|^{-3})$.
	\\ 
	\ref{sec:types}
	&  \begin{tabular}[c]{@{}l@{}}
	Establish the upper bound $|x - \alpha^{(j)}y| < c |t|^{-1} |y|^{-3}$ for a solution of type $j$;\\
	 a solution $(x,y)$ is of type $j$ if and only if $(-y,x)$ is of type $j+2 \pmod{4}$.
	 \end{tabular}
	\\
	\ref{sec:lowerbound}
	& Establish lower bounds of the shape $|t|^k/c < |y|$ using Padé approximations.
	\\
	\ref{sec:hypergeom}
	& Provide known tools and describe hypergeometric method.
	\\
	\ref{sec:proofOfProp}
	& Prove Proposition~\ref{prop:approx} with the hypergeometric method.
	\\
	\ref{sec:proofOfThm}
	& \begin{tabular}[c]{@{}l@{}}
	Prove Theorem~\ref{thm:main100} by combining Proposition~\ref{prop:approx} and the lower bounds \\
	from Sections \ref{sec:types} and \ref{sec:lowerbound}.
	 \end{tabular}
	\\
	\ref{sec:cors}
	& \begin{tabular}[c]{@{}l@{}}
	Prove Corollaries \ref{cor:lin} and \ref{cor:eps} by generalizing results from Sections \ref{sec:types} and \ref{sec:lowerbound}\\
	 and combining them with Proposition~\ref{prop:approx}.
	\end{tabular}
	\\
\hline
\end{tblr}
\end{table}

Finally, let us note that one could in principle solve inequality~\eqref{ineq:main} completely, i.e.\ also for parameters $|t|<100$, using Baker's method. 
However, in the quartic case it is not completely obvious how to find good general independent units, so solving the large number of equations requires some extra attention. Since the present paper is already rather long, the resolution for $|t|\leq 100$ is planned for future work. Some solutions for small $t$ are mentioned in Remark~\ref{rem:smallSols}.

\section{Irreducibility and solution of $F_t(x,y)=0$}\label{sec:irred}

First, we determine for which parameters $t$ the polynomial $f_t$ is reducible over $\Qd$. Note that for solving $F_t(X,Y)=0$, we only need to know whether $f_t(X)$ has a root in $\Qd$, but reducibility may be of independent interest.

\begin{lem}\label{lem:irred}
Let $d$ be a positive square-free integer and $t \in \qi$. Then the polynomial
\[
	f_t(X)
	%:= F_t(X,1)
	= X^4 -t X^3 - 6 X^2 + t X + 1
\]
is reducible over $\Q(\sqrt{-d})$ if and only if
\[
	t \in 
	\pm \{ 
	0,
	3,
	3i \pm 1,
	4i,
	5i,
	3 \sqrt{-2},
	2 \sqrt{-3},
	\frac{5 \sqrt{-3} \pm 3}{2},
	\sqrt{-7},
	\frac{3 \sqrt{-7} \pm 1}{2},
	\sqrt{-15}
	\}.
\]
Moreover, the polynomial $f_t(X)$ has a root in $\Qd$ if and only if $t=\pm 4i$.
\end{lem}
\begin{proof}
First, note that for $t = \pm 4i$ we have $f_t(X) = (X\mp i)^4$. Let us from now on assume $t \neq \pm 4i$.

Next, we determine the shape of the roots of $f_t$.
Let $\phi$ be the rational map $\phi \colon z \mapsto (1-z)/(1+z)$. 
We have $\phi^4 = \id$ and one can check by a straight forward computation that if $\alpha$ is a root of $f_t$, 
then also $\phi(\alpha)$ is a root of $f_t$.
Thus we have the roots 
\begin{equation}\label{eq:roots}
	\alpha, \quad
	\phi(\alpha)=\frac{\alpha-1}{\alpha+1},\quad 
	\phi^2(\alpha)=-\frac{1}{\alpha}, \quad
	\phi^3(\alpha)=-\frac{\alpha+1}{\alpha -1}.
\end{equation}
Now we check whether these roots are pairwise distinct. Since $\phi$ is cyclic of order 4, it suffices to check if $\alpha \neq \phi^2(\alpha)$. Assume that $\alpha = \phi^2(\alpha) = -1/\alpha$. Then $\alpha = \pm i$. Plugging into $f_t(X)$, we obtain $0 = f_t(\pm i) = 8 \pm 2 t i$,
which implies $t = \pm 4i$, which we excluded. Thus, we may from now on assume that $\alpha \neq \pm i$ and that the four roots of $f_t$ can be written as in \eqref{eq:roots}. 

Next, we check that $f_t$ has no roots in $\Qd$. Assume that $\alpha \in \Qd$. Then since $\phi$ is a rational map, all roots lie in $\Qd$. Moreover, being roots of the monic polynomial $f_t$, they are all algebraic integers. Since $\qi$ is integrally closed, all roots in fact lie in $\qi$. But $\alpha\in \qi$ and $-1/\alpha\in \qi$ implies that $\alpha$ is a unit in $\qi$. Therefore, we only need to check all possible units. 
Units in imaginary quadratic integer rings are always roots of unity, so we only need to check if $\alpha \in \{\pm 1, \pm i, \pm \zeta_6, \pm \zeta_6^2\}$ can be a root of $f_t$. The values $\pm 1$ are not possible because $f_t(\pm 1) = -4\neq 0$. The case $\alpha = \pm i$ has already been handled above. For $\alpha = \pm \zeta_6$, we plug into $f_t$ and solving $f_t(\zeta_6)=0$ for $t$, we obtain that $t= \pm7\sqrt{-3}/3$, which is not an imaginary quadratic integer. Also $f_t(\pm\zeta_6^2)=0$ implies that $t = \pm 7\sqrt{-3}/3$, which we are not interested in. Thus, $f_t$ has a root in $\Qd$ if and only if $t=\pm 4i$.

Finally, we need to check for which $t$ the polynomial $f_t$ factors into two irreducible polynomial over $\Qd$. Assume that 
\[
	f_t(X) 
	= (X^2 + a X + b)(X^2 + c X + d),
\]
with $a,b,c,d\in \Qd$ and that the two factors are irreducible. Assume that $\alpha$ is a root of the first polynomial. Then because of the structure of the roots, we may assume without loss of generality that the second root of $X^2 + aX +b$ is either $\phi(\alpha)$ or $\phi^2(\alpha)$. 
%(Note that in the case of $\alpha$ and $\phi^3(\alpha)$ being roots of the same polynomial we can rename $\alpha := \phi^3(\alpha)$ and we are in the first case). 
We consider these two cases separately.

\case{1} The roots of $X^2 + aX +b$ are $\alpha$ and $\phi(\alpha) = (\alpha - 1)/(\alpha + 1)$. 
Then we have $b = \alpha \cdot  (\alpha - 1)/(\alpha + 1)$, which implies $\alpha^2 - (b+1) \alpha - b = 0$, i.e.\ $\alpha$ is a root of $X^2 - (b+1) X - b$ with $b\in \Qd$. This contradicts the fact that the unique minimal polynomial of $\alpha$ is $X^2 + aX +b$.

\case{2} The roots of $X^2 + aX +b$ are $\alpha$ and $\phi^2(\alpha) = -1/\alpha$. Then we have $b = \alpha \cdot (-1/\alpha) = -1$ and consequently also $d = -1$. 
Thus we get
\begin{align*}\label{eq:irredtwofactors}
	X^4 -t X^3 - 6 X^2 + t X + 1
	&= (X^2 + a X - 1)(X^2 + c X - 1)\\
	&= X^4 + (a+c) X^3 + (ac - 2)X^2 - (a + c) X +1 \nonumber
\end{align*}
and comparing coefficients we obtain $a+c = -t$ and $ac = -4$. By Vietá's formula, these two equations imply that $a$ and $c$ are the roots of the polynomial $X^2 +tX - 4 = 0$. This implies that $a$ and $c$ are integral over $\qi$ and therefore in $\qi$. Moreover, since $ac=-4$, we have that $|a|^2$ and $|c|^2$ both divide 16. In particular, $|a|$ is bounded by 4 and we can list all such imaginary quadratic integers (see the \nameref{sec:appendix} for how to find them in a systematic way). Then for each $a\neq 0$ we check whether $|a|^2$ divides 16 and if so, we compute $c=-4/a$. Then, if $c$ is integral 
%(we check whether $|c|^2$ and $c+\conj{c}$ are integers)
, we compute $t = -(a+c)$.
Running these computations in Sage \cite{sagemath} takes less than a second and yields exactly the list of exceptional $t$'s that is stated in the lemma.
Note that if we consider $-a$ instead of $a$, we end up with $-c$ and $-t$, so for the computations it suffices to consider only $a$'s that are either positive or have positive imaginary part.
The Sage code for this proof can be found \href{https://cocalc.com/IngridVukusic/QuarticFamily/Irreducibility}{here}.
\end{proof}

Now we immediately get the following result: 

\begin{lem}\label{lem:zero}
Let $d$ be a positive square-free integer, $t \in \qi$ and
$t\neq \pm 4i$.
%\[
%	t \notin 
%	\pm \{ 
%	0,
%	3,
%	1 \pm 3i,
%	4i,
%	5i,
%	3 \sqrt{-2},
%	2 \sqrt{-3},
%	\sqrt{-7},
%	\frac{3 \pm 5 \sqrt{-3}}{2},
%	\sqrt{-15},
%	\frac{1 \pm 3 \sqrt{-7}}{2}
%	\}.
%\]
Then the equation
\begin{equation}\label{eq:zero}
	F_t(X,Y) 
	= X^4 -t X^3 Y - 6 X^2 Y^2 + t X Y^3 + Y^4
	= 0
	\quad \text{in } X,Y \in \qi
\end{equation}
has only the trivial solution $(x,y)=(0,0)$.
\end{lem}
\begin{proof}
Let $(x,y)\in \qi$ be a solution to \eqref{eq:zero}. If $y=0$, then it is easy to see that $x=0$, i.e.\ $(x,y)$ is the trivial solution. Otherwise, we have 
$0 = F_t(x,y) = y^4 f_t(x/y)$, i.e.\ $x/y \in \Qd$ is a root of $f_t$, which is impossible by Lemma~\ref{lem:irred}.
\end{proof}

Thus, in order to solve $|F_t(X,Y)|\leq 1$, we can from now on focus on the equation $F_t(X,Y) = \mu$, where $\mu$ is a root of unity in $\qi$.

\section{Small solutions}\label{sec:smallsols}

For technical reasons we will want to assume $\min\{|x|,|y|\} \geq 3$ later in the proof. Therefore, we now find all solutions for all $|t|\geq 100$ with $\min\{|x|,|y|\} < 3$. The proof is based on the idea in \cite[Proof of Lemma 5]{HeubergerPethoTichy2002}.

\begin{lem}\label{lem:smallSols}
Let $d$ be a positive square-free integer, $t\in \qi$ with $|t|\geq 100$ 
and $\mu \in \qi^\times$.
Let $(x,y)\in \qi$ be a solution to the equation
\begin{equation}\label{eq:unity}
	F_t(X,Y) = \mu.
\end{equation}
If $\min\{|x|,|y|\} < 3$, then $(x,y)$ is trivial, i.e.\ of the shape $(\xi, 0)$ or $(0, \xi)$, where $\xi$ is a root of unity.
\end{lem}
\begin{proof}
If $x=0$ or $y=0$, we see immediately that $(x,y)$ is trivial by plugging into \eqref{eq:unity}. Let us from now on assume that $x\neq 0$ and $y\neq 0$.

Since the solutions $(x,y)$ and $(-y,x)$ are equivalent, we may assume without loss of generality that $|x|\leq|y|$ and in particular $0 < |x| < 3$. Moreover, since $(x,y)$ and $(-x,-y)$ are equivalent, we may assume without loss of generality that either the imaginary part of $x$ is
$\Im(x) > 0$ or that $x\in \Z$ and $x>0$.
There are only finitely many such imaginary quadratic integers $x$ with $0<|x| < 3$ and we can list them, see the \nameref{sec:appendix} for more details.
For each of these values of $x$, we now describe how to find all $(x,y)$ that are a solution to \eqref{eq:unity} for some $t$ and $\mu$.

Assume that $(x,y)$ is a solution to~\eqref{eq:unity} with $0<|x| < 3$.
Then~\eqref{eq:unity} implies that
\begin{equation}\label{eq:smally_rewrittenmain}
	y(y^3 + t x y^2 - 6x^2 y - tx^3) = \mu - x^4.
\end{equation}
We consider two cases.

\case{1} $\mu - x^4 \neq 0$.
Since all elements are imaginary quadratic integers, we get from \eqref{eq:smally_rewrittenmain} that
\begin{align*}
	|y|
	= |y| \cdot 1
	\leq |y| \cdot |y^3 + t x y^2 - 6x^2 y - tx^3|
	= |\mu -x^4|
	 \leq |\mu| + |x|^4
	 % < 1 + 3^4 = 82, 
\end{align*}
and thus $|y|\leq 1 + |x|^4$. 

\case{2} $\mu - x^4 = 0$. 
Since we are assuming $y\neq 0$, equation~\eqref{eq:smally_rewrittenmain} implies
\[
	y^3 + t x y^2 - 6x^2 y - tx^3 = 0,
\]
which is equivalent to
\begin{equation}\label{eq:smally_prod}
	(tx + y)(y-x)(x+y) = 5 x^2 y.
\end{equation}
Noting that $|x|=|\mu|^{1/4}=1$, we obtain from \eqref{eq:smally_prod} that
\[
	5 |y|
	= 5 |x|^2 |y|	
	= |tx + y| |y-x| |x+y|
	\geq 1 \cdot (|y|-|x|)(|y|-|x|)
	= (|y|-1)^2, 
\]
which implies $|y|<6.86$.

Thus, in both cases we have a small upper bound on $|y|$ and there are only finitely many such imaginary quadratic integers $y$ and we can list them. Note that, if $x\notin \Z$, then the quadratic integers $y$ are in the fixed imaginary quadratic field $\Q(x)$. Moreover, one can check that $f_t(x,-y) = f_{-t}(x,y)$ for any $t,x,y$, and we can therefore restrict our search to $y$ with either  $\Im(y)>0$ or $y\in \Z$ and $y>0$.

Then we only need to plug all such $(x,y)$ into \eqref{eq:unity} and compute 
\[
	t = \frac{x^4 - 6x^2y^2 + y^4 - \mu}{x^3 y - xy^3}
\]
and check if it might be a quadratic integer.
Here we need to check all possible units $\mu$ that are in the numberfield of $x$ and $y$. 
Note that if $(x,y)\in \Z^2$, then we check all roots of unity $\pm 1$, $\pm i$, $\pm \zeta_6$, $\pm \zeta_6^2$.

Now if $t$ is a quadratic integer 
%(we check whether $|t|^2$ and $t+\conj{t}$ are integers),
and $|t|\geq 100$, then we have found a non-trivial solution of interest. Otherwise, $(x,y)$ cannot be a non-trivial solution of a considered equation.

Doing all the computations with Sage \cite{sagemath} takes a few seconds on a usual pc and reveals no non-trivial solutions (see  \href{https://cocalc.com/IngridVukusic/QuarticFamily/SmallSolutions}{here}).
\end{proof}

\begin{rem}\label{rem:smallSols}
If we drop the assumption $|t| \geq 100$ in the above described computer search, we get nontrivial solutions for 
\[
	t \in 
		\pm \{
		4i,
		3 \sqrt{-2},
		2 \sqrt{-3}
	\}
	\cup
	\pm \{
		1,
		4,
		\frac{\sqrt{-3} \pm 1}{2},
		\sqrt{-17}
	\}
	.
\]
For the $t$'s in the first set we have by Lemma~\ref{lem:irred} that $f_t$ is reducible, while for the $t$'s in the second set it is irreducible.
\end{rem}

\section{Approximation of the roots}\label{sec:approx}

%\begin{equation*}
%	f_t(X)
%	:= F_t(X,1)
%	= X^4 -t X^3 - 6 X^2 + t X + 1.
%\end{equation*}

Let $\alpha=\alpha^{(0)}, \alpha^{(1)}, \alpha^{(2)}, \alpha^{(3)}$ be the roots of $f_t$. As described in the proof of Lemma~\ref{lem:irred}, we may write
\begin{equation}\label{eq:alphas}
	\alpha^{(0)}= \alpha, \quad
	\alpha^{(1)}= \frac{\alpha-1}{\alpha+1},\quad 
	\alpha^{(2)}=-\frac{1}{\alpha}, \quad 
	\alpha^{(3)}=-\frac{\alpha+1}{\alpha-1}.
\end{equation}

We now compute asymptotic estimates for the roots of $f_t$. This can be done e.g.\ with Sage \cite{sagemath}, substituting $1/t=:s$ and doing the computations in the ring of power series with variable $s$. One can approximate $\alpha$ by starting at $x_0=0$ and applying several steps of Newton's Method. Then the other approximations can be obtained from the formulas in \eqref{eq:alphas}.
The error terms follow from an application of Rouché's Theorem, see the proof below.

We use the following $L$-notation: For functions $h,k$ we write $h(z) = L(k(z))$ if $|h(z)|\leq k(z)$ for all $z\geq 100$. 

\begin{lem}\label{lem:roots-approx}
Let $t\in \C^\times$ with $|t|\geq 100$. Then the roots of $f_t(X)= X^4 -t X^3 - 6 X^2 + t X + 1$ are approximated by
\begin{align*}
	\alpha^{(0)} &= - \frac{1}{t} +  L\left( \frac{5.01}{|t|^3}\right), 
	& \alpha^{(2)} &= t + L\left( \frac{5.02}{|t|} \right), \\
	\alpha^{(1)} &= - 1 + L\left(\frac{2.16}{|t|}\right),
	& \alpha^{(3)} &= 1  + L\left(\frac{2.16}{|t|}\right).
\end{align*}
\end{lem}
\begin{proof}
We want to prove the approximation for $\alpha^{(0)}$ via Rouché's Theorem. 
Let us set $h(z):=f(-\frac{1}{t}+z)$. The goal is to show that $h(z)$ has a root in the disc with origin 0 and radius $5.01|t|^{-3}$.
We check that $|h(0)| < |h(z)-h(0)|$ for any $z$ with $|z|=5.01 |t|^{-3}$. On the one hand we have 
\begin{align*}
	|h(0)|
	= \left| f(-\frac{1}{t}) \right|
	= \left|- \frac{5}{t^2} + \frac{1}{t^4} \right|
	\leq \frac{5.0001}{|t|^2},
\end{align*}
where we used $|t|\geq 100$ for the estimate. On the other hand we have
\begin{align*}
	|h(z)-h(0)|
	&= \left| f(-\frac{1}{t} + z) - f(-\frac{1}{t})\right| \\
	&= \left| z^4 
		+ \left( -t - \frac{4}{t}\right) z^3 
		+ \left( - 3 + \frac{6}{t^2} \right) z^2 		
		+ \left( t + \frac{9}{t} - \frac{4}{t^3} \right) z \right|\\
	&\geq \left| t z \right| 
		- \left|z^4  + \left( -t - \frac{4}{t}\right) z^3 
			+ \left( - 3 + \frac{6}{t^2} \right) z^2 		
			+ \left(\frac{9}{t} - \frac{4}{t^3} \right) z\right|\\
	&>  |t| \cdot 5.01 |t|^{-3} - 50|t|^{-4}
	> \frac{5.0001}{|t|^2},
\end{align*}
where we used $|z|=5.01 |t|^{-3}$ and $|t|\geq 100$.
Overall, we have obtained
\[
	|h(0)| < |h(z)-h(0)|
	\quad \text{for }
	|z|=5.01 |t|^{-3}.
\]
%for $|z|=48|t|^{-5}$.
Now note that the function $h_1(z):=h(z)-h(0)$ has a root with $|z|<5.01 |t|^{-3}$, namely $z=0$. Moreover, $h_1(z)$ and the constant function $h(0)$ are holomorphic. Thus Rouché's Theorem tells us that the function $h_1(z)+h(0)=h(z)=f(-\frac{1}{t} + z)$ has a root with $|z|<5.01 |t|^{-3}$. This immediately implies that $f$ has a root with value $- \frac{1}{t} + L\left( \frac{5.01}{|t|^3}\right)$.

The other approximations can be checked analogously, see the \href{https://cocalc.com/IngridVukusic/QuarticFamily/Rouche}{code} for details. 
\end{proof}

\section{Types of solutions}\label{sec:types}

As in Section~\ref{sec:smallsols}, let $d$ be a positive square-free integer and we continue focusing on the Thue equations of the type
\begin{equation}\label{eq:unity2}
	F_t(X,Y) = \mu,
\end{equation}
where $t \in \qi$ and $\mu \in \qi^\times$.

In general, when solving Thue equations, one of the tricks is to use the following fact: A solution $(x,y)$ usually corresponds to an extremely good approximation to one of the roots of the related univariate polynomial. 

Let $(x,y)\in \qi^2$ be a solution to equation~\eqref{eq:unity2}. Then we define
\[
	\beta^{(i)}
	:= x - \alpha^{(i)}y,
	\quad \text{for } i = 0,1,2,3.
\]
We say that $(x,y)$ is a solution of {\it type} $j$ if  
\[
	|\beta^{(j)}|=\min_{0\leq i \leq 3} |\beta^{(i)}|.
\]

The next lemma quantifies how good such an approximation is.

\begin{lem}\label{lem:betajsmall}
Let, $t \in \qi$, $|t|\geq 100$ and $\mu \in \qi^\times$. 
Let $(x,y)\in \qi^2$ be a non-trivial solution to the equation 
\begin{equation*}
	F_t(X,Y) = \mu.
\end{equation*}
If $(x,y)$ is of type $j$, then
\begin{equation*}
	|x - \alpha^{(j)} y|
	=|\beta^{(j)}|
	<\frac{8.86}{|t|\cdot |y|^3}.
\end{equation*}
\end{lem}
\begin{proof}
Let $(x,y)$ be a solution of type $j$. From 
\[
	|y||\alpha^{(i)}-\alpha^{(j)}|
	= | \alpha^{(i)} y - x + x - \alpha^{(j)}y|
	\leq |\beta^{(i)}| + |\beta^{(j)}|\leq 2|\beta^{(i)}|
\]
and 
\[
	\prod_{i=0}^{3} \beta^{(i)}	
	=\prod_{i=0}^{3}(x-\alpha^{(i)}y)
	=F_t(x,y)=\mu,
\]
we conclude that 
\begin{equation*}\label{eq:betaj-leq}
	|\beta^{(j)}|
	=\frac{1}{\prod_{i\neq j}|\beta^{(i)}|}
	\leq \frac{2^3}{|y|^3 \prod_{i\neq j}|\alpha^{(j)}-\alpha^{(i)}|}.
%	=\frac{1}{|y|^3}\cdot\frac{8}{(f^{4}_t)^{\prime}(\alpha^{(j)})}
\end{equation*}

From Lemma~\ref{lem:roots-approx} we see (assuming $|t|\geq 100$) that the difference between any two roots is at least $0.96$.

Moreover, the difference between $\alpha^{(2)}$ and any other root is at least $0.98|t|$. Thus we obtain
\begin{equation*}	
	|\beta^{(j)}|
	\leq \frac{1}{|y|^3}\cdot\frac{8}{0.96^2\cdot 0.98|t|}
	<\frac{8.86}{|t|\cdot |y|^3}.
\end{equation*}

\end{proof}

Finally, we describe how the types of equivalent solutions $(x,y)$ and $(-y,x)$ relate.

\begin{lem}\label{lem:typewlog}
Let $t \in \qi$, $|t|\geq 100$ and $\mu \in \qi^\times$. 
Let $(x,y)\in \qi^2$ be a solution to the equation 
\begin{equation*}
	F_t(X,Y) = \mu
\end{equation*}
with $\min\{|x|,|y|\} \geq 3$. Then $(x,y)$ is of type $j$ if and only if $(-y,x)$ is a solution of type $j+2\pmod{4}$.
\end{lem}
\begin{proof}
Let $|t|\geq 100$. From Lemma~\ref{lem:betajsmall} it follows that
\[
	\left| \alpha^{(j)} - \frac{x}{y} \right|
	< \frac{8.86}{|t| \cdot |y|^4}
	< 0.09
\]
for any non-trivial solution $(x,y)$ of any type $j$. Since by Lemma~\ref{lem:roots-approx} the distance between any two distinct roots is at least $0.96$, we can say that any non-trivial solution is of type $j$ if and only if $|\alpha^{(j)}-x/y|<0.48$. 

Now let $(x,y)$ be a non-trivial solution of type $j$. Then from Lemma~\ref{lem:betajsmall} we get (with $j+2$ computed in modulo 4 arithmetic) that
\begin{align*}
	\left| \alpha^{(j+2)} - \frac{-y}{x} \right|
	= \left| - \frac{1}{\alpha^{(j)}} + \frac{y}{x} \right|
	= \left| \frac{-x + \alpha^{(j)} y}{\alpha^{(j)}x} \right|
	< \frac{8.86}{|t| \cdot |y|^3 \cdot |x| \cdot |\alpha^{(j)}|}.
\end{align*}
Since by Lemma~\ref{lem:roots-approx} we have $|\alpha^{(j)}|>0.94 |t|^{-1}$, we obtain
\begin{align*}
	\left| \alpha^{(j+2)} - \frac{-y}{x} \right|
	< \frac{9.43}{|x|\cdot |y|^3}.
\end{align*}
With $\min\{|x|,|y|\} \geq 3$ we obtain 
\begin{align*}
	\left| \alpha^{(j+2)} - \frac{-y}{x} \right|
	< 0.48,
\end{align*}
which implies that $(-y,x)$ is indeed of type $j+2 \pmod{4}$.
\end{proof}

In view of Lemma~\ref{lem:typewlog}, since $(x,y)$ and $(-y,x)$ are equivalent solutions, it is enough to consider only solutions of type 0 or of type 3 in order to prove Theorem~\ref{thm:main100}.

\section{Finding lower bounds for $|y|$}\label{sec:lowerbound}

In this section we prove the following lower bounds for $|y|$. Note that we could push the absolute lower bound even higher, but the goal is only to contradict the upper bound that will follow from Proposition~\ref{prop:approx}.

\begin{lem}\label{lem:lbfory}
Let, $t \in \qi$, $|t|\geq 100$ and $\mu \in \qi^\times$. 
Let $(x,y)\in \qi^2$ be a solution to the equation 
\begin{equation}\label{eq:unity-steps}
	F_t(X,Y) = \mu.
\end{equation}
Assume that $\min\{|x|,|y|\}\geq 3$ and that $(x,y)$ is a solution of type 0 or 3.
Then we have
\begin{equation}\label{eq:lem:lbfory_rel}
	|y|
	> 0.44 |t|
\end{equation}
and moreover
\begin{equation}\label{eq:lem:lbfory_abs}
	|y|
	> 1.047 \cdot 10^{13}.
\end{equation}
\end{lem}
 
In fact, we will do several steps and  prove bounds of the shape
\[
	|y|
	> \frac{|t|^{k}}{c},
\]
for larger and larger integers $k$. This is based on the ideas in \cite[Section 5]{Ziegler2005}. First, we consider the case where the solution is of type 0. Afterwards, we do analogous computations for the type 3 case.

\subsection{Lower bound for $y$ of type 0}\label{sec:lowerbound-type0}

Let $|t|\geq 100$ and $(x,y)\in \qi^2$ be a solution of type 0 and $\min\{|x|,|y|\}\geq 3$.

\begin{step}{1}\label{step:type0-1}
We combine Lemma~\ref{lem:betajsmall} and the approximation of $\alpha=\alpha^{(0)}$ from Lemma~\ref{lem:roots-approx}:
\begin{align*}
	|x|-\frac{1.01}{|t|}|y|
	\leq \left| x - \left(  - \frac{1}{t} + L\left( \frac{5.01}{|t|^3}\right) \right)y \right|
	= |x - \alpha y| 
	< \frac{8.86}{|t|\cdot |y|^3}.
\end{align*}
Since we are assuming $\min\{|x|,|y|\}\geq 3$, this implies
\begin{align*}
	3	
	\leq |x|
	< \frac{1.01}{|t|}|y| + \frac{8.86}{|t|\cdot |y|^3}
	< \frac{1.12}{|t|} |y|
\end{align*}
and we get
\begin{equation*}
	|y|
	> 2.67 \cdot |t|.
\end{equation*}
In particular, we have proven equation~\eqref{eq:lem:lbfory_rel} in the type 0 case.
\end{step}

\begin{step}{2}\label{step:type0-2}
As in the previous step, we combine Lemma~\ref{lem:roots-approx} and Lemma~\ref{lem:betajsmall}, but now we multiply the expressions with an extra factor $|t|$:
\begin{align*}
	|tx + y|-\frac{5.01}{|t|^2}|y|
	\leq \left| t x - \left(  - 1 + L\left( \frac{5.01}{|t|^2}\right) \right)y \right|
	= |t x - t \alpha y| 
	< \frac{8.86}{|y|^3}.
\end{align*}
If $tx+y=0$, then plugging $y=-tx$ into equation \eqref{eq:unity-steps} yields $x^4 (1-5t^2)=\mu$. Considering the absolute values, this immediately leads to a contradiction for $t \neq 0$.
Thus we may assume that $tx+y$ is a non-zero element of $\qi$, which means that it has absolute value at least 1 and we obtain from the above inequality
\begin{align*}
	1	
	\leq |tx + y|
	< \frac{5.01}{|t|^2}|y| + \frac{8.86}{|y|^3}
	< \frac{5.01}{|t|^2}|y| + \frac{8.86}{( 2.67 \cdot |t|)^4}|y|
	< \frac{5.02}{|t|^2}|y|,
\end{align*}
which implies
\begin{equation*}
	|y|
	> \frac{|t|^2}{5.02}.
\end{equation*}
\end{step}

The trick in Step~\ref{step:type0-2} to multiply the equation with $|t|$ worked because of the gap in the exponents of $|t|$ in  the main and the $L$-term in the approximation of $\alpha$. This is now exhausted and we need a new idea. We use the idea from \cite[Proof of Lemma 7]{Ziegler2005}, which is to replace the series expansion of $\alpha$ by Padé approximations. Moreover, we will need a higher precision approximation of $\alpha$. Analogously to Lemma~\ref{lem:roots-approx} one can prove (see \href{https://cocalc.com/IngridVukusic/QuarticFamily/Rouche}{code}) the following.

\begin{lem}\label{lem:betterApprox}
	Let $B(1/t)$ be the approximation to $\alpha$ obtained by applying $15$ steps of Newton's method and with precision $O(|t|^{-31})$, i.e.
\[
	B(1/t) 
	= -t^{-1} + 5 t^{-3} - 46 t^{-5} + 509 t^{-7} + \dots 
		+ 1821914025180536 t^{-29}.
\]
Then we have 
\[	
	|\alpha - B(1/t)| < \frac{2.71 \cdot 10^{16}}{|t|^{31}}.
\] 
\end{lem}

To compute Padé approximations $A(1/t)=U(1/t)/V(1/t)$ to $B(1/t)$ we can use Sage's power series method \verb|pade()| with the variable $s:=1/t$.

\begin{step}{3}\label{step:type0-3}
In Step 3 we compute the Padé approximation where the degrees of $U$ and $V$ are bounded by 2.
We obtain 
\[
	\alpha 
	\approx B(1/t) 
	\approx A(1/t)
	= \frac{U(1/t)}{V(1/t)} 
	= \frac{-\frac{1}{t}}{\frac{5}{t^2} + 1}.
\]
The constant that we will obtain in this step will mostly depend on the quality of this approximation. Therefore, we first compute the approximation error in the following sense:
\begin{align*}
	|U(1/t) - B(1/t) V(1/t)|
	= \left| \frac{21}{t^5} - \frac{279}{t^7} + \dots  \right|
	< \frac{21.028}{|t|^5},
\end{align*}
where the dots stand for a finite expression, which was actually computed and then estimated using $|t|\geq 100$.

Next, we get $x$ and $y$ involved and compute some other approximations. First, we use Lemma~\ref{lem:betajsmall} and the bound from Step~\ref{step:type0-2}:
\begin{equation*}\label{eq:step3-xalphay}
	|x - \alpha y|
	< \frac{8.86}{|t|\cdot |y|^3}
	< \frac{8.86}{|t|(|t|^2/5.02)^{4}}|y|
	< \frac{5627}{|t|^9}|y|.
\end{equation*}
With this estimate and Lemma~\ref{lem:betterApprox} we get
\begin{align*}
	|x - B(1/t) y|
	\leq |x - \alpha y| + |\alpha - B(1/t)|\cdot |y|
	< \frac{5627}{|t|^9}|y| + \frac{2.71 \cdot 10^{16}}{|t|^{31}} |y|.
	%\leq \frac{5628}{|t|^9}|y|.
\end{align*}
We need one more intermediate estimate:
\[
	|V(1/t)|
	= \left| 1 + \frac{5}{t^2} \right|
	< 1.01.
\]
Finally, we combine these estimates to find a useful upper bound for an imaginary quadratic integer:
\begin{align} \label{eq:step3-1}
	|(t^2+5)x + t y|
	&=|t^2 V(1/t) x - t^2 U(1/t)y|\\
	&= \left| t^2 V(1/t) x - t^2 \left( V(1/t)B(1/t) + L\left(\frac{21.028}{|t|^5} \right)\right) y \right| \nonumber \\
	&\leq  \frac{21.028}{|t|^3}|y| +  |t|^2 |V(1/t)| |x - B(1/t)y| \nonumber \\
	&< \frac{21.028}{|t|^3}|y| +  |t|^2 \cdot 1.01 \cdot 
		\left( \frac{5627}{|t|^9} + \frac{2.71 \cdot 10^{16}}{|t|^{31}} \right) |y| \nonumber \\
	&< \frac{21.028}{|t|^3}|y| +  \left(\frac{5684}{|t|^7} + \frac{2.74 \cdot 10^{16}}{|t|^{29}}\right)|y| \nonumber \\
	&< \frac{21.03}{|t|^3}|y|. \nonumber
\end{align}

Assume for a moment that the expression on the left vanishes. Then we have $x=-t/(t^2+5)y$. Computing $F_t(x,y)$ and using $|y|> |t|^2/5.02$ we obtain
\begin{align} \label{eq:Ftabsvalue}
	|F_t(x,y)|
 	&=\left| F_t \left( \frac{-t}{t^2 + 5} y , y \right)\right|
 	= \left| f_t \left( \frac{-t}{t^2 + 5} \right)\right| \cdot |y|^4\\
	&= \left| \frac{21 t^4 + 225 t^2 + 625 }{(t^2 + 5)^4}
		\right| \cdot |y|^4
	> \frac{21|t|^4 - 225|t|^2 - 625}{(|t|^2 +5)^4}
		\left(\frac{|t|^2}{5.02}\right)^4. \nonumber
\end{align}
For $|t|\geq 100$ the last expression is larger than 1, which is a contradiction.

Thus we may assume that the expression bounded in \eqref{eq:step3-1} is non-zero. Since it is clearly an imaginary quadratic integer, it has absolute value at least 1. This yields
\[
	1 < \frac{21.03}{|t|^3}|y|,
\]
which implies
\begin{equation*}
	|y|
	> \frac{|t|^3}{21.03}.
\end{equation*}
\end{step}

\begin{steps}{4--9}\label{steps:type0-4-9}
At each Step $k$ ($4 \leq k \leq 9$) we start with a bound of the form 
\begin{equation*}
	|y|
	> \frac{|t|^{k-1}}{c_0},
\end{equation*}
where in Step 4 we have $c_0 = 21.03$.
Then we compute a Padé approximation $\alpha \approx A(1/t) = U(1/t)/V(1/t)$ where $U,V \in \Z[X]$ are polynomials of degree at most $k-1$. 

Analogously to Step~\ref{step:type0-3}, it we can then compute a constant $c_1$ such that
\[
	|U(1/t) - B(1/t) V(1/t)|
	\leq \frac{c_1}{|t|^{2 k -1}}.
\]
Next, we can estimate
\[
	|x - \alpha y|
	< \frac{8.86}{|t|\cdot |y|^3}
	< \frac{8.86}{|t|(|t|^{k-1}/c_0)^{4}}|y|
	= \frac{c_2}{|t|^{4k-3}}|y|,
\]
with $c_2 = 8.86 \cdot c_0^4$, and
\[
	|x - B(1/t)y|
	< \left( \frac{c_2}{|t|^{4k-3}} + \frac{2.71 \cdot 10^{16}}{|t|^{31}} \right) |y|.
	%\leq \frac{c_3}{|t|^{4k-3}}|y|.
\]
Finally, we compute an upper bound 
\[
	|V(1/t)| \leq c_3.	
\]
Since we are assuming $|t|\geq 100$, it turns out that $d$ is always roughly the size of the constant term in $V$.

Then, by an analogous computation to \eqref{eq:step3-1} we obtain
\begin{align}\label{eq:stepk-main}
	|t^{k-1} V(1/t) x - t^{k-1} U(1/t) y|
	< \left( \frac{c_1}{|t|^{k}} + \frac{c_2 c_3}{|t|^{3k-2}} + \frac{2.71 \cdot 10^{16} \cdot c_3}{|t|^{31 + 1 - k}}\right) |y|
	\leq \frac{c}{|t|^{k}} |y|,
\end{align}
where $c = c_1 + c_2 c_3 \cdot 100^{-(2k-2)} + 2.71 \cdot 10^{16} \cdot c_3 \cdot 100^{-(32 - 2k)}$, which is roughly of the size of $c_1$.

To finish the step, we only need to check that the left hand side cannot vanish. 
If it did, we would have $x= A(1/t) y$. To show that this is impossible, we do an analogous computation to \eqref{eq:Ftabsvalue}:
\begin{align*}
	|F_t(x,y)|
 	=\left| f_t \left( \frac{t^{k-1}U(1/t)}{t^{k-1}V(1/t)} \right)\right| \cdot |y|^4
	&= \left| \frac{P(t)}{(t^{k-1}V(1/t))^4}
		\right| \cdot |y|^4 \\
	&> \frac{|P(t)|}{|t^{k-1}V(1/t)|^4}
		\left(\frac{|t|^{k-1}}{c_0}\right)^4. \nonumber
\end{align*}
It turns out that $P(t) = F_t(t^{k-1}U(1/t), t^{k-1}V(1/t))$ is always a polynomial of degree $2k-2$ and $t^{k-1}V(1/t)$ is a polynomial of degree $k-1$. Thus, we have a lower bound of the order $|t|^{2k-2}$. Computing the coefficients and estimating using $|t|\geq 100$, one can indeed show in every step that $|F_t(A(1/t)y,y)|>1$, a contradiction.

Thus we may assume that the left hand side of \eqref{eq:stepk-main} is at least 1, which implies
\begin{equation}\label{eq:lbyfromsteps}
	|y|
	> \frac{|t|^{k}}{c}
\end{equation}
and we set $c_0 = c$ for the next step.

Always rounding up with a 4-digit precision when computing $c$, we obtain the values presented in Table~\ref{tab:steps} (see \href{https://cocalc.com/IngridVukusic/QuarticFamily/LowerBoundForY}{code}). Moreover, the table shows the lower bounds for $|y|$ that follow immediately from setting $|t|=100$ in \eqref{eq:lbyfromsteps}. The lower bound for $|y|$ obtained in Step 9 proves equation~\eqref{eq:lem:lbfory_abs} of Lemma~\ref{lem:lbfory} in the type 0 case.
\end{steps}

\begin{table}[h]
\caption{Constants obtained in Steps 4--9 for type 0}\label{tab:steps}
\resizebox{\textwidth}{!} {
\begin{tabular}{ccccccccc}
\hline
$k$ & $4$ & $5$ & $6$ & $7$ & $8$ & $9$ & $10$ & $11 $ \\ \hline
$c$ & $429.8$ & $2436$ & $4210$ & $1.863 \cdot 10^5$ & $3.242 \cdot 10^{6}$ & $5.915 \cdot 10^{6}$ & $8.066 \cdot 10^{7} $ & $4.726 \cdot 10^{8} $ \\
$|y| > \ldots$ & $2.327 \cdot 10^5$ & $4.105 \cdot 10^{6}$ & $2.375 \cdot 10^{8}$ & $5.369 \cdot 10^{8}$ & $3.084 \cdot 10^{9}$ & $1.691 \cdot 10^{11}$ & $1.240 \cdot 10^{12} $ & $2.116 \cdot 10^{13} $ \\ \hline
\end{tabular}
}
\end{table}

\subsection{Lower bound for $y$ of type 3}\label{sec:lowerbound-type3}

Now assume that $(x,y)$ with $\min\{|x|,|y|\}\geq 3$ is a solution of type 3. We proceed analogously to the above subsection.

\begin{step}{1}\label{step:type3-1}
We combine Lemma~\ref{lem:betajsmall} and the approximation of $\alpha^{(3)}$ from Lemma~\ref{lem:roots-approx}:
\begin{align*}
	|x - y| - \frac{2.16}{|t|}|y|
	\leq \left| x - \left(  1 + L\left( \frac{2.16}{|t|}\right) \right)y \right|
	= |x - \alpha^{(3)} y| 
	< \frac{8.86}{|t|\cdot |y|^3}.
\end{align*}
Since we are assuming $|y|\geq 3$, this implies
\begin{align*}
	|x - y|
	\leq \frac{2.16}{|t|}|y| + \frac{8.86}{|t|\cdot |y|^3}
	< \frac{2.27}{|t|} |y|.
\end{align*}
If $x = y$, then plugging into $F_t(x,y) = \mu$ yields $-4x^4 = \mu$, which is impossible. Thus we have $|x-y|\geq 1$ and the above inequality implies
\begin{equation*}
	|y|
	> \frac{|t|}{2.27}
	> 0.44 |t|.
\end{equation*}
Thus we have proven equation~\eqref{eq:lem:lbfory_rel} in the type 3 case.
\end{step}

For the further steps, we need a higher precision approximation for $\alpha^{(3)}$. We can compute it in Sage via $\alpha^{(3)}=-(\alpha+1)/(\alpha-1)$, where we use an approximation to $\alpha$ obtained with Newton's method.
Analogously to Lemma~\ref{lem:roots-approx} and Lemma~\ref{lem:betterApprox}, one can prove (see \href{https://cocalc.com/IngridVukusic/QuarticFamily/Rouche}{code}) the following.

\begin{lem}\label{lem:betterApproxalpha3}
Let $B_3(1/t)$ be the approximation to $\alpha^{(3)}$ given by
\[
	B_3(1/t) 
	= 1 -2t^{-1} + 2 t^{-2} + 8 t^{-3} - 18 t^{-4} + \dots 
		+ 1435829041889280 t^{-29}.
\]
Then we have 
\[	
	|\alpha - B_3(1/t)| < \frac{9.84 \cdot 10^{15}}{|t|^{30}}.
\] 
\end{lem}

The remaining steps are analogous to Step~\ref{step:type0-3} and Steps~\ref{steps:type0-4-9} from the previous subsection. We only need to replace $\alpha$ by $\alpha^{(3)}$, $B(1/t)$ by $B_3(1/t)$ and $2.71 \cdot 10^{16}/|t|^{31}$ by ${9.84 \cdot 10^{15}}/{|t|^{30}}$.

\begin{steps}{2--10}\label{steps:type3-2-10}
We start with $k=2$ and $c_0 = 2.27$.
The results from each step (see \href{https://cocalc.com/IngridVukusic/QuarticFamily/LowerBoundForY}{code}) are presented in Table~\ref{tab:stepsalpha3}. The lower bound for $|y|$ obtained in Step 10 proves equation~\eqref{eq:lem:lbfory_abs} of Lemma~\ref{lem:lbfory} in the type 3 case. Thus we have completed the proof of Lemma~\ref{lem:lbfory}.
\end{steps}

\begin{table}[h]
\caption{Constants obtained in Steps 4--9 for type 3}\label{tab:stepsalpha3}
\resizebox{\textwidth}{!} {
\begin{tabular}{ccccccccccc}
\hline
$k$ & $2$ & $3$ & $4$ & $5$ & $6$ & $7$ & $8$ & $9$ & $10$ & $11$ \\ \hline 
$c$ & $10.14$ & $42.48$ & $868.0$ & $4921$ & $8503$ & $3.762 \cdot 10^5$ & $6.549 \cdot 10^{6}$ & $1.195 \cdot 10^{7}$ & $1.629 \cdot 10^{8}$ & $9.547 \cdot 10^{8}$ \\ 
$|y| > \ldots$ & $986.2$ & $23540$ & $1.152 \cdot 10^5$ & $2.032 \cdot 10^{6}$ & $1.176 \cdot 10^{8}$ & $2.658 \cdot 10^{8}$ & $1.527 \cdot 10^{9}$ & $8.370 \cdot 10^{10}$ & $6.138 \cdot 10^{11}$ & $1.047 \cdot 10^{13} $ \\ \hline
\end{tabular}
}
\end{table}

\section{The hypergeometric method}\label{sec:hypergeom}

The hypergeometric method has its name from the hypergeometric function, which in this context is used to construct very good approximations of a root $\alpha$ of $f_t(X)$. With these approximations one can then obtain an effective irrationality measure for $\alpha$ by the following elementary but ingenious lemma. We have borrowed it from
\cite[Lemma 2.7]{Heuberger2006}, which is a generalization of
\cite[Lemma 2.8]{ChenVoutier1997}.
However, the idea is much older and goes back to Thue and Siegel, see \cite{Chudnovsky1983} for a historic overview.

\begin{knownlem}\label{lem:approx-irrat}
Let $\alpha\in\mathbb{C}$ and suppose that there exist real numbers $k_0,l_0> 0$ and $E, Q > 1$ such that for all positive integers $r$ there are integers $p_r , q_r \in \qi$ with $|q_r| < k_0 Q^r$ and $|q_r\alpha- p_r|\leq l_0 E^{-r}$ satisfying $p_rq_{r+1}\neq p_{r+1}q_r$ for all $r$. 
Then for any integers $p, q \in \qi$ with $|q|\geq 1/(2l_0)$, we have
\[
	\left|\alpha-\frac{p}{q}\right|
	>\frac{1}{c|q|^{\kappa+1}}, 
	\quad \text{where} \quad 
	c=2k_0Q(2l_0E)^\kappa
	\quad \text{and} \quad 
	\kappa=\frac{\log Q}{\log E}.
\]
\end{knownlem}

Now we discuss how to obtain such good approximations to a root $\alpha$ of $f_t(X)$.
We will construct sequences of polynomials $\pA_r, \pB_r \in \Q(\sqrt{-d})[X]$ such that $|\alpha \pA_r(\xi) - \pB_r(\xi)|$ is very small if $\xi$ is close to $\alpha$. This will almost give us the approximations that we need for the application of Lemma~\ref{lem:approx-irrat}. We will just need to choose $\xi$ appropriately and clear denominators with numbers $M_r$, so that $q_r = M_r \pA_r(\xi)$ and $p_r = M_r \pB_r(\xi)$ are indeed in $\qi$.

Let $\hf$ denote the classic \textit{hypergeometric function}. For positive integers $n,r$ set
\begin{align*}
\chi_{n,r}(X)
	&= \hf \left( -r, -r-\frac{1}{n}; 1-\frac{1}{n}; X \right)
	\in \Q[X], \\
\chi^*_{n,r}(X,Y)
	&= Y^r\chi_{n,r} \left(\frac{X}{Y}\right) 
	\in\Q[X,Y].
\end{align*}
Note that $\chi_{n,r}$ is a polynomial of degree $r$ and $\chi^{*}_{n,r}$ is its homogenization. 

The basis for constructing $\pA_r$ and $\pB_r$ is Thue's ``Fundamentaltheorem'' \cite{Thue1910}. We use the version that is stated in \cite[Lemma 2.1]{Heuberger2006}.

\begin{knownlem}[Thue]\label{lem:Thue}
Let $K$ be a field of characteristic $0$, $\pP \in K[X]$ be a square-free polynomial of degree $n\geq 2$ and assume that there is a square-free quadratic polynomial $\pU \in K[X]$ such that
\begin{equation}\label{eq:Thue-main}
	\pU \pP''- (n-1) \pU' \pP' + \frac{n(n-1)}{2} \pU'' \pP
	= 0
\end{equation}
holds, where the prime denotes differentiation with respect to the indeterminate $X$. 
We set $\lambda =\frac{1}{4}disc(\pU),$ where $\disc(\pU) = \pU'^{2}- 2 \pU \pU'' \in K$ is the discriminant of $\pU$.
We define the polynomials contained in $K(\sqrt{\lambda})[X]$ by 
\begin{align*}
\pY 
	&= 2\pU \pP'-n\pU'\pP, \\
\pa
	&= \frac{n^2-1}{6}(\sqrt{\lambda} \pU'+2\lambda),
	& \pc 
	&= \frac{n^2-1}{6}(\sqrt{\lambda}(\pU'X-2\pU)+2\lambda X),\\
\pb &=\frac{n^2-1}{6}(\sqrt{\lambda} \pU'-2\lambda),
	& \pd
	&= \frac{n^2-1}{6}(\sqrt{\lambda}(\pU'X-2\pU)-2\lambda X),\\	
\pu &= \frac{1}{2} \left(\frac{\pY}{2n\sqrt{\lambda}} - \pP \right),
	& \pz
	&= \frac{1}{2} \left(\frac{\pY}{2n\sqrt{\lambda}} + \pP \right).
\end{align*}
Finally, for later Lemmas, we set $\pw = \pz / \pu \in K(\sqrt{\lambda})(X)$.

Then for $r\in\N$, the polynomials $\pA_r, \pB_r$ given by
\begin{align*}
(\sqrt{\lambda})^r \pA_r 
	&= \pa \chi^{*}_{n,r}(\pz,\pu) - \pb \chi^{*}_{n,r}(\pu,\pz),\\
(\sqrt{\lambda})^r \pB_r 
	&= \pc \chi^{*}_{n,r}(\pz,\pu) - \pd \chi^{*}_{n,r}(\pu,\pz)
\end{align*}
are elements of the polynomial ring $K[X]$ over $K$. For every root $\alpha$ of $\pP$, the polynomial
\[
	\pC_r
	= \alpha \pA_r - \pB_r
\]
is divisible by $(X - \alpha)^{2r+1}$.
\end{knownlem}

When constructing the approximations of Lemma~\ref{lem:approx-irrat} we will have to make sure that $p_rq_{r+1}\neq p_{r+1}q_r$ for all $r$. We will use the following lemma \cite[Lemma 2.2]{Heuberger2006}, which is a special case of \cite[Lemma 2.7]{ChenVoutier1997}.

\begin{knownlem}\label{lem:distinctApprox}
Let $\pA_r$, $\pB_r$, $\pP$ and $\pU$ be defined as in Lemma~\ref{lem:Thue}. If $\pU(\xi)\pP(\xi)\neq 0$ for a given $\xi\in\mathbb{C}$, then for all positive integers $r$, we have  
\[
	\pA_{r+1}(\xi)\pB_{r}(\xi)\neq \pA_{r}(\xi)\pB_{r+1}(\xi).
\]
\end{knownlem}

The approximations $\pA_r(\xi) / \pB_r(\xi)$ will be very good; however, in general $\pA_r(\xi), \pB_r(\xi)$ won't be integers (even if $\xi$ is), mainly because of the denominators coming from $\chi_{n,r}$. To clear these denominators, we will use the following lemma, which is a result by Lettl et al. \cite[Proposition 2c]{LettlPethoVoutier1999}. In our version of the lemma we have rewritten and slightly weakened the two inequalities when rounding, to make them more naturally applicable in our context.
Note that in the application in the next section, $\xi$ will indeed have the shape $1- 8 x$ as in the lemma.

\begin{knownlem}[Lettl et al.]\label{lem:denominators}
Let $r$ be a positive integer,  
$\Delta_{4,r}$ be the least common multiple of the denominators of the coefficients of $\chi_{4,r}$ and let
$N_{4,r}$ be the greatest common divisor of the numerators of the coefficients of $\chi_{4,r}(1- 8 X)$.
Then $\Delta_{4,r}/N_{4,r}\chi_{4,r}(1 - 8 X)$ is a polynomial with integer coefficients and we have 
\[
	\frac{2^{r+2}\Delta_{4,r}}{N_{4,r}}\cdot\frac{\Gamma(3/4)r!}{\Gamma(r+3/4)}
	<3.32\cdot 1.35^r
	\quad \text{and} \quad 
	\frac{2^{4r+3}\Delta_{4,r}}{N_{4,r}}\cdot\frac{\Gamma(r+5/4)}{\Gamma(1/4)r!}
	<1.6\cdot 10.7^r.
\]
\end{knownlem}

With the help of the factors $\Delta_{4,r}/N_{4,r}$ from the above lemma, we will construct the actual approximations $p_r/q_r$ to our root $\alpha$. 
In order to apply Lemma~\ref{lem:approx-irrat} we will first have to estimate $|q_r|$. On the one hand, we will use the estimates from the above Lemma~\ref{lem:denominators}, on the other hand, we will need estimates for $\chi_{4,r}$. The next lemma follows from \cite[Lemma 7.3b]{Voutier2010} with $m=1$, $n=4$, $z=1$ and $u = u/z = w$. Note that we cannot use \cite[Lemma 4]{LettlPethoVoutier1999} because we won't have $|w|=1$ since in our case $t$ is not real.

\begin{knownlem}[Voutier]\label{lem:chi}
Let $w$ be a complex number with $|1-w|<1$ and $|1-w^{-1}|<1$ and let $r$ be a positive integer.
Then
\begin{equation*}
	|\chi_{4,r}(w)|
	\leq \frac{\Gamma(3/4) 2^{r+1} r!}{\Gamma(r+3/4)}\cdot (1+|w|)^r.
\end{equation*}
\end{knownlem}

Finally, in order to estimate $|q_r \alpha - p_r|$, which will be a multiple of $\pC_r(\xi)$ from Lemma~\ref{lem:Thue}, we will use the following Lemma. It is a combination of Lemmas 2.3 and 2.4 in  \cite{Heuberger2006}. Note that the former is proven by Chen and Voutier~\cite[Lemma 2.3]{ChenVoutier1997} and that the requirement $\xi \neq 0$ was removed in a recent version of the paper, see the Addendum of \cite{ChenVoutier1997_2018}. 
For roots of complex numbers we will agree to choose the root where the argument has the smallest absolute value (and in case of ambiguity with the positive value), i.e.\ in the lemma below we have $-\pi /n < \arg (w(\xi)^{1/n})\leq \pi/n$.

\begin{knownlem}[Chen and Voutier; Heuberger]\label{lem:R}
Let $n\geq 2$ and $r$ be positive integers and let $\alpha$, $\lambda$, $\pa(X)$, $\pb(X)$, $\pc(X)$, $\pd(X)$, $\pC_r(X)$, $\pu(X)$, $\pz(X)$ and $\pw(X)$ be as in Lemma~\ref{lem:Thue}. Further, let $\xi$ be a complex number such that $|\pw(\xi) -1 | < 1 $. 
Then we can write
\begin{align}
	(\sqrt{\lambda})^r \pC_r(\xi) \label{eq:Cr}
	= &( 
		\alpha (\pa(\xi)\pw(\xi)^{1/n} -  \pb(\xi))
		-  (\pc(\xi) \pw(\xi)^{1/n} - \pd(\xi))
	  )
	  \cdot \chi_{n,r}^*(\pu(\xi), \pz(\xi))\\
	  &- (\alpha \pa(\xi)-\pc(\xi)) \cdot \pu(\xi)^r \cdot R_{n,r}(\pw(\xi)),\nonumber
\end{align}
with the estimate
\[
	|R_{n,r}(\pw(\xi))|
	\leq \frac{\Gamma(r+1+1/n)}{r!4^r\Gamma(1/n)}
		\cdot \frac{|\pw(\xi)-1|^{2r+1}}{(1-|\pw(\xi)-1|)^{r+1-1/n}}.
\]
\end{knownlem}

%\section{Obtaining an effective irrationality measure}
\section{Proof of Proposition \ref{prop:approx} (irrationality measure)}\label{sec:proofOfProp}

In this section, we prove Proposition~\ref{prop:approx}, i.e.\ we find an effective measure of irrationality for the roots $\alpha=\alpha^{(0)}$ and $\alpha^{(3)}$ of the polynomial $f_t(X)= X^4 - tX^3 - 6X^2 + tX + 1$, with an imaginary quadratic integer $t$ with $|t|\geq 100$.

We start by determining the quantities defined in Lemma~\ref{lem:Thue}. 
First, put
\begin{align*}
\pP(X) 
= f_t(X)
= X^4 - tX^3 - 6X^2 + tX + 1.
\end{align*}
Then one can check that there indeed exists a square-free quadratic polynomial $\pU$ that satisfies the differential equation~\eqref{eq:Thue-main}, namely the polynomial
\[
	\pU(X)=X^{2}+1.
\]
We have $\disc(\pU)=-4$ and we set $\lambda=-1$, as well as
\begin{align*}
\pY(X) 
	&= 2tX^4 + 32X^3 - 12tX^2 - 32X + 2t,\\
\pa(X)
	&= 5iX - 5,
	& \pc(X)
	&= - 5X - 5i, \\
\pb(X)
	&= 5iX + 5, 
	& \pd(X) 
	&= 5X -5i,\\
\pu(X)
	&= - \frac{it + 4}{8}(X+i)^4,
	& \pz(X)
	&= \frac{-it + 4}{8}(X-i)^4,\\
\pw(X)
	&= \frac{it-4}{it+4} \cdot \frac{(X-i)^4}{(X+i)^4}.
\end{align*}
For $\pA_r$ and $\pB_r$ we get the formulas 
\begin{align*}
\pA_r 
	&= (-i)^r (\pa \chi^{*}_{n,r}(\pz,\pu) - \pb \chi^{*}_{n,r}(\pu,\pz))
	= (-i)^r (\pa \pu^r \chi_{n,r}(\pw) - \pb \pz^r \chi_{n,r}(\pw^{-1})),\\
\pB_r 
	&= (-i)^r (\pc \chi^{*}_{n,r}(\pz,\pu) - \pd \chi^{*}_{n,r}(\pu,\pz))
	= (-i)^r (\pc \pu^r \chi_{n,r}(\pw) - \pd \pz^r \chi_{n,r}(\pw^{-1})).	
\end{align*}
Lemma~\ref{lem:Thue} then implies that $\pC_r = \alpha A_r - B_r$ is divisible by $(X-\alpha)^{2r+1}$, i.e.\ we can expect $|\alpha A_r - B_r|$ to be very small if evaluated at a $\xi$ close to the root $\alpha$. Thus, if we want to approximate $\alpha=\alpha^{(0)}$, we have to choose a $\xi$ close to $\alpha^{(0)}$. In view of Lemma~\ref{lem:roots-approx}, $\xi=0$ will be a good choice for $\xi$. Similarly, if we want to approximate $\alpha^{(3)}$ we will choose $\xi=1$. 

\subsection{Irrationality measure for $\alpha^{(0)}$}

Let us first focus on $\alpha^{(0)}\approx 0$.
We compute:
\begin{align*}	
\pa(0) &= -5, \quad
	& \pc(0) &= -5i,\\
\pb(0) &= 5, \quad
	& \pd(0) &= -5i, \\
\pu(0) &= - \frac{it + 4}{8} =: -\frac{u}{8}, \quad
	& \pz(0) &= \frac{-it + 4}{8} =: - \frac{z}{8} , \\
\pw(0) &= \frac{it - 4}{it + 4} = 1 - \frac{8}{u} :=w.
\end{align*}
where we defined 
\[	
	u=it+4,
	\quad 
	z=it-4 \quad \text{and} \quad
	w= \frac{it-4}{it+4}.
\]
Noting that $\pw(0)^{-1} = w^{-1} = \frac{it+4}{it-4} = 1 + \frac{8}{z}$ we obtain from the definitions
\begin{align*}
\pA_r(0)
	&= (-i)^r \left(
		-5 \left(-\frac{u}{8}\right)^r \chi_{4,r}(1-\frac{8}{u})
		- 5 \left(-\frac{z}{8}\right)^r \chi_{4,r}(1+\frac{8}{z})
	\right) \\
	&= -5 i^r 8^{-r} \left(
		u^r \chi_{4,r}(1-\frac{8}{u}) + z^r \chi_{4,r}(1+\frac{8}{z})
		\right),\\
\pB_r(0)
	&= (-i)^r \left( 
		- 5i \left(-\frac{u}{8}\right)^r \chi_{4,r}(1-\frac{8}{u})
		+ 5i \left(-\frac{z}{8}\right)^r \chi_{4,r}(1+\frac{8}{z})
	\right) \\
		&= -5 i^{r+1} 8^{-r} \left(
		u^r \chi_{4,r}(1-\frac{8}{u}) - z^r \chi_{4,r}(1+\frac{8}{z})
		\right).
\end{align*}

In order to obtain algebraic integers, we clear the denominators of $\pA_r(0)$ and $\pB_r(0)$ with the notation from Lemma~\ref{lem:denominators}: We set 
\[
	M_r(0) = \frac{8^r}{5} \cdot \frac{\Delta_{4,r}}{N_{4,r}} 
	\quad \text{and} \quad
	p_r(0) = M_r(0) \pB(0), \quad 
	q_r(0) = M_r(0) \pA(0).
\]
Then we have
\begin{align*}
	p_r(0)
	&= -i^{r+1} \frac{\Delta_{4,r}}{N_{4,r}} \left(
		u^r \chi_{4,r}(1-8u^{-1}) - z^r \chi_{4,r}(1+8z^{-1})
		\right), \\
	q_r(0)
	&=	
	-i^r \frac{\Delta_{4,r}}{N_{4,r}} \left(
		u^r \chi_{4,r}(1-8u^{-1}) + z^r \chi_{4,r}(1+8z^{-1})
		\right)
	.
\end{align*}
We check that $p_r(0), q_r(0)$ are in $\qi$ for any $r$. First, recall that by Lemma~\ref{lem:Thue} the polynomials $\pA_r(X), \pB_r(X)$ are polynomials with coefficients in $\Q(\sqrt{-d})$ and therefore $\pA_r(0), \pB_r(0) \in \Q(\sqrt{-d})$. Since $M_r(0) \in \Q$, we clearly have $p_r(0), q_r(0) \in \Q(\sqrt{-d})$. Now we check that $p_r(0), q_r(0)$ are algebraic integers. The factors $-i^r$ and $-i^{r+1}$ are algebraic integers. By the definition of $\Delta_{4,r}$ and $N_{4,r}$ in Lemma~\ref{lem:denominators} we have that $\Delta_{4,r}/N_{4,r}\cdot \chi_{4,r}(1 - 8 X)$ is a polynomial with integer coefficients of degree $r$. Therefore, since $u,z$ are algebraic integers,
we see that $u^r \cdot \Delta_{4,r}/N_{4,r}\cdot \chi_{4,r}(1 - 8 u^{-1})$ and $z^r \cdot \Delta_{4,r}/N_{4,r}\cdot \chi_{4,r}(1 - 8(-z^{-1}))$ are algebraic integers as well. Thus $p_r(0)$ and $q_r(0)$ are algebraic integers and since they are in $\Q(\sqrt{-d})$, we have indeed $p_r(0),q_r(0)\in \qi$.

In order to apply Lemma~\ref{lem:approx-irrat} we need to estimate $|q_r(0)|$ and $|q_r(0) \alpha - p_r(0)|$ from above.
Recall that we have set
\[
	w 
	= \pw(0)
	= 1 - 8 u^{-1}
	= (1 + 8z^{-1})^{-1}
	= \frac{it - 4}{it + 4}.
\]
To check that these equalities  hold, see the computation of $\pw(0)$ and the computations below that. 
Now we can write
\begin{align*}
	p_r(0)
	&= -i^{r+1} \frac{\Delta_{4,r}}{N_{4,r}} \left(
		u^r \chi_{4,r}(w) - z^r \chi_{4,r}(w^{-1})
		\right), \\
	q_r(0)
	&= -i^r \frac{\Delta_{4,r}}{N_{4,r}} \left(
		u^r \chi_{4,r}(w) + z^r \chi_{4,r}(w^{-1})
		\right).
\end{align*}
In order to apply Lemma~\ref{lem:chi}, we check that $|1-w|<1$ and $|1-w^{-1}|<1$:
\begin{align}\label{eq:1-w}
	|1-w|
%	= \left| 1 - \frac{it + 4}{it - 4} \right|
	= \left| \frac{- 8}{it - 4} \right|
	\leq \frac{8}{|t| - 4}
	<1
\end{align}
and analogously we obtain $|1-w^{-1}|<1$.
Now we use Lemma~\ref{lem:chi}, the fact that $|u|=|it+4|\leq |t|+4$ and $|z|=|it-4|\leq |t|+4$: 
\begin{align*}
	|q_r(0)|
	&\leq \frac{\Delta_{4,r}}{N_{4,r}} \left(
		|u|^r |\chi_{4,r}(w)| + |z|^r |\chi_{4,r}(w^{-1})|
		\right) \\
	&\leq \frac{\Delta_{4,r}}{N_{4,r}} (|t|+4)^r \left(
		\frac{\Gamma(3/4) 2^{r+1} r!}{\Gamma(r+3/4)}\cdot (1+|w|)^r	
		+ 	\frac{\Gamma(3/4) 2^{r+1} r!}{\Gamma(r+3/4)}\cdot (1+|w^{-1}|)^r	
		\right)\\
	&= \frac{\Delta_{4,r}}{N_{4,r}} \frac{\Gamma(3/4) 2^{r+1} r!}{\Gamma(r+3/4)}(|t|+4)^r \left(
		(1+|w|)^r + (1+|w^{-1}|)^r	
		\right).
\end{align*}
Next we use the estimates $|w|,|w^{-1}| \leq 1 + 8/(|t|-4) < 1.09$ 
and $|t|+4 \leq 1.04|t|$ for $|t|\geq 100$, as well as Lemma~\ref{lem:denominators}, obtaining

\begin{align*}
	|q_r(0)|
	&\leq \frac{2^{r+1} \Delta_{4,r}}{N_{4,r}} \frac{\Gamma(3/4) r!}{\Gamma(r+3/4)} (|t|+4)^r \cdot 2 \cdot 2.09^r\\
		&\leq \frac{2^{r+2} \Delta_{4,r}}{N_{4,r}} \frac{\Gamma(3/4) r!}{\Gamma(r+3/4)} (1.04 |t|)^r \cdot 2.09^r\\
	&\leq 3.32\cdot 1.35^r \cdot (1.04 |t|)^r \cdot 2.09^r \\
	&< 3.32 \cdot (2.94|t|)^r.
\end{align*}
Thus we can set $k_0=3.32$ and $Q=2.94|t|$ in Lemma~\ref{lem:approx-irrat}. Next, we need to find an upper bound for the estimation error
\[
	|\alpha q_r(0) - p_r(0)|
	= |\alpha M_r(0)\pA_r(0)  - M_r(0) \pB_r(0)|
	= M_r(0) |\pC_r(0)|.
\] 
We want to apply Lemma~\ref{lem:R} and we first show that the coefficient $\alpha (\pa(0)\pw(0)^{1/4} -  \pb(0))
		-  (\pc(0) \pw(0)^{1/4} - \pd(0))$
of $\chi_{4,r}^*(\pu(0),\pz(0))$ in \eqref{eq:Cr} vanishes. To that aim we verify that the expression
\begin{align}
	\frac{\pc(0) \pw(0)^{1/4} - \pd(0)}{\pa(0) \pw(0)^{1/4} - \pb(0)} \label{eq:expression-root0}
	= i \cdot \frac{\left(\frac{it-4}{it+4}\right)^{1/4} - 1}{ \left(\frac{it-4}{it+4}\right)^{1/4} + 1}
\end{align}
is a root of $f_t(X)$. This can be done with a straightforward computation. Moreover, since $(it - 4)/(it+4) \approx 1$ for large $|t|$,
the absolute value of the expression in \eqref{eq:expression-root0} is very small for $|t|\geq 100$. Therefore, the above expression must be exactly $\alpha=\alpha^{(0)}$ and the first summand in \eqref{eq:Cr} vanishes. Thus we obtain from Lemma~\ref{lem:R} that
\begin{align*}
	|\pC_r(0)| 
		&\leq | \alpha \pa(0) - \pc(0)| \cdot |\pu(0)|^r 
			\cdot \frac{\Gamma(r+1+1/4)}{r!4^r\Gamma(1/4)}
			\cdot \frac{|\pw(0)-1|^{2r+1}}{(1-|\pw(0)-1|)^{r+1-1/4}} \\
		&= |-5 \alpha +5i| \cdot \left(\frac{|it+4|}{8}\right)^r
			\cdot \frac{\Gamma(r+5/4)}{r! 4^r \Gamma(1/4)}
			\cdot \frac{|w-1|^{2r+1}}{(1-|w-1|)^{r+3/4}}.
\end{align*}
Recall from \eqref{eq:1-w} that $|1-w|\leq 8/(|t|-4)$, which moreover implies 
\[
	1-|w-1|
	\geq 1- \frac{8}{|t|-4}
	= \frac{|t|-12}{|t|-4}.
\]
We continue estimating $|\pC_r(0)|$:
\begin{align*}
	|\pC_r(0)| 
		&\leq 5(|\alpha|+1) 2^{-5r} (|t|+4)^r \frac{\Gamma(r+5/4)}{r! \Gamma(1/4)} 
			\cdot \left(\frac{8}{|t|-4}\right)^{2r+1} 
			\cdot \left( \frac{|t|-4}{|t|-12} \right)^{r+3/4}\\
		&= 5(|\alpha|+1) 2^{r+3} (|t| +4)^r \frac{\Gamma(r+5/4)}{r! \Gamma(1/4)}  (|t|-4)^{-r-1/4} (|t|-12)^{-r-3/4}.
\end{align*}
Now note that for $|t|\geq 100$ we have
\begin{align*}
	|\alpha|\leq 1/|t| + 5.01/|t|^3 < 0.02, \quad
	|t|+4 \leq 1.04 |t|, \quad
	|t|-4 \geq 0.96 |t|, \quad
	|t|-12 \geq 0.88 |t|.
\end{align*}
Thus we obtain
\begin{align*}
	|\pC_r(0)| 
		&< 5 \cdot 1.02 \cdot 2^{r+3} \frac{\Gamma(r+5/4)}{r! \Gamma(1/4)} 
			\cdot |t|^{-r-1} 1.04^r \cdot 0.96^{-r-1/4} \cdot 0.88^{-r-3/4} \\
		&< 5 \cdot 2^{r+3}\cdot \frac{\Gamma(r+5/4)}{r! \Gamma(1/4)} 
			\cdot 1.14 \cdot |t|^{-1} \cdot \left(\frac{1.24}{|t|}\right)^r.
\end{align*}
Using this estimate, the definition of $M_r(0)$ and Lemma~\ref{lem:denominators} we obtain
\begin{align*}
	M_r(0) |\pC_r(0)|
		&< \frac{8^r}{5} \cdot \frac{\Delta_{4,r}}{N_{4,r}} 
			\cdot 5 \cdot 2^{r+3}\cdot \frac{\Gamma(r+5/4)}{r! \Gamma(1/4)} 
			\cdot 1.14 \cdot |t|^{-1}\cdot \left(\frac{1.24}{|t|}\right)^r\\
		&= \frac{2^{4r+3} \Delta_{4,r}}{N_{4,r}} 
			 \cdot \frac{\Gamma(r+5/4)}{\Gamma(1/4)r!} 
			\cdot 1.14 \cdot |t|^{-1}\cdot \left(\frac{1.24}{|t|}\right)^r\\
		&<  1.6 \cdot 10.7^r
			\cdot 1.14 \cdot |t|^{-1}\cdot \left(\frac{1.24}{|t|}\right)^r\\
		&< 1.83 |t|^{-1}\cdot \left(\frac{13.27}{|t|}\right)^{r}.
\end{align*}
Thus we have proven that $|\alpha q_r(0) - p_r(0)| < 1.83 |t|^{-1}\cdot (|t|/13.27)^{-r}$ and we can set $l_0=1.83/|t|$ and $E=|t|/13.27$ in Lemma~\ref{lem:approx-irrat}. Let us sum up what we have achieved so far: Assuming $|t|\geq 100$, for the root $\alpha$ we have found $p_r(0),q_r(0)\in \qi$ for all positive integers $r$, such that $|q_r(0)|<k_0 Q^r$ and $|\alpha q_r(0) - p_r(0)| \leq l_0 E^{-r}$ with
\[
	k_0 = 3.32, \quad
	Q = 2.94|t|, \quad
	l_0 = 1.83/|t|, \quad
	E= |t|/13.27.
\]
Now we only need to check that $p_r(0)q_{r+1}(0) \neq p_{r+1}(0) q_r(0)$ for all $r$. This follows immediately from Lemma~\ref{lem:distinctApprox} as $\pU(0)=\pP(0)=1\neq 0$, and the fact that $M_r(0)\neq 0$ for all $r$. 
Thus we can apply Lemma~\ref{lem:approx-irrat} with
\begin{align*}
	\kappa	
	= \frac{\log Q}{\log E}
	&= \frac{\log |t| + \log 2.94}{\log |t| -\log 13.27}
	\leq \frac{\log |t| + 1.08}{\log |t| - 2.59}.
\end{align*}
Note that $\kappa<3$ for $|t|\geq 84$.
Moreover, we have
\begin{align*}
c 
	= 2k_0Q(2l_0E)^\kappa
	&= 2\cdot 3.32 \cdot 2.94|t| (2\cdot 1.83/|t| \cdot |t|/13.27)^\kappa\\
	&< 19.53 |t| \cdot 0.28^\kappa
	< 19.53 |t| \cdot 0.28
	< 5.47|t|.
\end{align*}
Finally, note that 
\[
	1/(2l_0)= 1/(2\cdot 1.83/|t|)
	< 0.28 |t|.
\]
Thus Lemma~\ref{lem:approx-irrat} yields
\[
	\left| \alpha - \frac{p}{q} \right|
	> \frac{1}{5.47 |t| \cdot |q|^{\kappa +1 }}
	\quad \text{with} \quad 
	\kappa = \frac{\log |t| + 1.08}{\log |t| - 2.59}
\]
for all $|q|\geq 0.28 |t|$.

%{5.27 ist unplausibel klein! ueberpruefen wo anders, kein Fehler gefunden}
%{kann $l_0$ auch groesser machen, dann ist die konstante im bruch schlechter aber dafuer gilt fuer mehr qs. Brauche genau das aber nicht.}
We have proven Proposition~\ref{prop:approx} for $j=0$. Note that the worse constant $15.48$ instead of $5.47$ will come from the type 3 case.

\subsection{Irrationality measure for $\alpha^{(3)}$}

Now we quickly repeat all computations for $\alpha^{(3)}\approx 1$.
\begin{align*}	
\pa(1) &= 5i - 5, \quad
	& \pc(1) &= -5i - 5,\\
\pb(1) &= 5i + 5, \quad
	& \pd(1) &= -5i + 5, \\
\pu(1) &= \frac{it + 4}{2} = \frac{u}{2}, \quad
	& \pz(1) &= \frac{it - 4}{2} = \frac{z}{2} , \\
\pw(1) &= \frac{it - 4}{it + 4} = 1 - \frac{8}{u} =w,
\end{align*}
where as before 
\[	
	u=it+4,
	\quad 
	z=it-4 \quad \text{and} \quad
	w= \frac{it-4}{it+4}
		= 1 - \frac{8}{u}
		= \left(1 + \frac{8}{z}\right)^{-1}.
\]
We continue with the computations as above:
\begin{align*}
\pA_r(1)
	&= (-i)^r \left( 
		(5i - 5) \left(\frac{u}{2}\right)^r \chi_{n,r}(1 - \frac{8}{u}) 
		- (5i + 5) \left(\frac{z}{2}\right)^r \chi_{n,r}(1 + \frac{8}{z})	
		\right)\\
	&= 5 (i-1) (-i)^r 2^{-r}  \left( 
		u^r \chi_{n,r}(1 - \frac{8}{u}) 
		+ i \cdot z^r \chi_{n,r}(1 + \frac{8}{z})	
		\right),\\
\pB_r(1)
	&= (-i)^r \left(
		(-5i -5) \left( \frac{u}{2} \right)^r \chi_{n,r}(1 - \frac{8}{u}) 
		- (-5i+5) \left( \frac{z}{2} \right)^r \chi_{n,r}(1 + \frac{8}{z})
		\right)\\
	&= -5 (i+1) (-i)^r 2^{-r} \left(
		 u^r \chi_{n,r}(1 - \frac{8}{u}) 
		- i \cdot z^r \chi_{n,r}(1 + \frac{8}{z})
		\right).	
\end{align*}
We set 
\[
	M_r(1) = \frac{2^r}{5} \cdot \frac{\Delta_{4,r}}{N_{4,r}} 
	\quad \text{and} \quad
	p_r(1) = M_r(1) \pB(1), \quad 
	q_r(1) = M_r(1) \pA(1)
\]
and obtain
\begin{align*}
	p_r(1)
	&= - (i+1) (-i)^r \frac{\Delta_{4,r}}{N_{4,r}} \left(
		u^r \chi_{4,r}(1-8u^{-1}) - i \cdot z^r \chi_{4,r}(1+8z^{-1})
		\right), \\
	q_r(1)
	&= (i-1) (-i)^r \frac{\Delta_{4,r}}{N_{4,r}} \left(
		u^r \chi_{4,r}(1-8u^{-1}) + i \cdot z^r \chi_{4,r}(1+8z^{-1})
		\right).
\end{align*}
The numbers $p_r(1), q_r(1)$ are in $\qi$ by the same arguments as for $p_r(0), q_r(0)$ above.

Next, we estimate $|q_r(1)|$. This is completely analogous to the estimate for $|q_r(0)|$ from above, except that we now have the additional factor $(i-1)$ with absolute value $\sqrt{2}$. Thus we end up with
\begin{align*}
	|q_r(0)|
	< \sqrt{2} \cdot 3.32 \cdot (2.94|t|)^r
	< 4.7 \cdot (2.94|t|)^r.
\end{align*}
Next, we find an upper bound for
\[
	|\alpha q_r(1) - p_r(1)|
	= |\alpha M_r(1)\pA_r(1)  - M_r(1) \pB_r(1)|
	= M_r(1) |\pC_r(1)|.
\] 
As above, one can check that
\[
	\frac{\pc(1) \pw(1)^{1/4} - \pd(1)}{\pa(1) \pw(1)^{1/4} - \pb(1)} 
	= \frac{\left(\frac{it-4}{it+4}\right)^{1/4} - i}
		   {-i \left(\frac{it-4}{it+4}\right)^{1/4} + 1}
\]
is a root of $f_t(X)$, which is close to $1$, and therefore must be equal to $\alpha^{(3)}$. Thus the first summand in \eqref{eq:Cr} vanishes and Lemma~\ref{lem:R} yields
\begin{align*}
	|\pC_r(1)| 
		&\leq | \alpha^{(3)} \pa(1) - \pc(1)| \cdot |\pu(1)|^r
			\cdot \frac{\Gamma(r+1+1/4)}{r!4^r\Gamma(1/4)}
			\cdot \frac{|\pw(1)-1|^{2r+1}}{(1-|\pw(1)-1|)^{r+1-1/4}} \\
		&= |5 (i-1) (\alpha^{(3)} - i)| \cdot \left(\frac{|it+4|}{2}\right)^r
			\cdot \frac{\Gamma(r+5/4)}{r! 4^r \Gamma(1/4)}
			\cdot \frac{|w-1|^{2r+1}}{(1-|w-1|)^{r+3/4}}.
\end{align*}
The rest of the estimation is completely analogous to that of $\pC_r(1)$, except we now have the extra factor $|i-1|=\sqrt{2}$. Moreover, instead of the factor from before $|- \alpha + i| \leq |\alpha| + 1 \leq 1.02$, we now have $|\alpha^{(3)} - i| \leq |1-i| + 2.16/|t| \leq \sqrt{2} + 2.16/100 \leq 1.44$. 
Thus we end up with 
\begin{align*}
	M_r(1) |\pC_r(1)|
	&\leq \frac{1.83\cdot \sqrt{2} \cdot 1.44}{1.02} \cdot |t|^{-1}\cdot \left(\frac{13.27}{|t|}\right)^{r}\\
	&\leq 3.66 |t|^{-1}\cdot \left(\frac{13.27}{|t|}\right)^{r}.
\end{align*}
Now we can set
\[
	k_0 = 4.7, \quad
	Q = 2.94|t|, \quad
	l_0 = 3.66/|t|, \quad
	E= |t|/13.27
\]
and we get the same $\kappa$ as before. For the constant $c$ we now have
\begin{align*}
c 
	= 2k_0Q(2l_0E)^\kappa
	&= 2\cdot 4.7 \cdot 2.94|t| (2\cdot 3.66/|t| \cdot |t|/13.27)^\kappa\\
	&\leq 27.64 |t| 0.56^\kappa
	\leq 27.64 |t| 0.56
	\leq 15.48|t|.
\end{align*}
Finally, note that 
\[
	1/(2l_0)= 1/(2\cdot 3.66/|t|)
	\leq 0.14 |t|.
\]
Lemma~\ref{lem:approx-irrat} finally yields
\[
	\left| \alpha - \frac{p}{q} \right|
	> \frac{1}{15.48 |t| \cdot |q|^{\kappa +1 }}
	\quad \text{with} \quad 
	\kappa = \frac{\log |t| + 1.08}{\log |t| - 2.59}
\]
for all $|q|\geq 0.14 |t|$.

We have thus proven Proposition~\ref{prop:approx} for $i=3$. Note that the stronger assumption $|q|\geq 0.28$ came from the case $i=0$.
Overall, Proposition~\ref{prop:approx} is now proven.

\section{Proof of Theorem \ref{thm:main100} (resolution of equation)}\label{sec:proofOfThm}

Since we have already solved $F_t(X,Y) = 0$ in Section~\ref{sec:irred}, in order to finish the proof of Theorem~\ref{thm:main100} we only need to solve the equations of the shape 
\begin{equation}\label{eq:unity-mainproof}
	F_t(X,Y) = \mu.
\end{equation}
Let $t\in \qi$ with $|t|\geq 100$, let $\mu \in \qi^\times$ and let $(x,y)\in \qi^2$ be non-trivial solution to equation~\eqref{eq:unity-mainproof}. In view of Lemma~\ref{lem:smallSols} we may assume $\min\{|x|,|y|\}\geq 3$.
By Lemma~\ref{lem:typewlog} we may assume without loss of generality that $(x,y)$ is either of type~0 or of type~3. 
From Lemma~\ref{lem:lbfory} we get that $|y|\geq 0.44 |t| \geq 0.28|t|$.
Now we can combine Proposition~\ref{prop:approx} and Lemma~\ref{lem:betajsmall}:
\begin{align*}
	\frac{1}{15.48 |t| \cdot |y|^{\kappa +1 }}
	<\left| \alpha - \frac{x}{y} \right| 
	< \frac{8.86}{|t|\cdot |y|^4}.
\end{align*}
This implies
\[
	|y|^{3-\kappa} 
	< 8.86 \cdot 15.48 
	< 137.16
\]
and with $\kappa < 2.83$
\[
	|y| 
	< 137.16^{1/0.17}
	< 3.74 \cdot 10^{12}.
\]
This contradicts the lower bound $|y| > 1.047 \cdot 10^{13}$ from Lemma~\ref{lem:lbfory}. Thus we have proven that there are no non-trivial solutions to equation \eqref{eq:unity-mainproof} for $|t| \geq 100$.

\begin{rem}
The proof of Theorem~\ref{thm:main100} works in principle as long as $\kappa < 3$, we just have to increase the lower bound for $|y|$ by pushing the Padé approximations in the proof of Lemma~\ref{lem:lbfory} further.
Since $\kappa < 3$ for $|t|\geq 84$, one could extend Theorem~\ref{thm:main100} to roughly $|t|\geq 84$ with the method used in this paper.
Moreover, one could slightly improve the result by estimating more carefully in the proof of Proposition~\ref{prop:approx}.
We refrained from this in favor of readability.
\end{rem}

\section{Proof of Corollaries \ref{cor:lin} and \ref{cor:eps}}\label{sec:cors}

Corollaries~\ref{cor:lin} and~\ref{cor:eps} are concerned with the inequalities $|F_t(X,Y)|\leq C |t|$ and $|F_t(X,Y)|\leq |t|^{2-\eps}$ respectively. As mentioned in Section~\ref{sec:results}, the Corollaries follow relatively quickly from Proposition~\ref{prop:approx} and in contrast to Theorem~\ref{thm:main100} the proofs make use of the fact that $\kappa$ can get arbitrarily close to 1.

Before proving the Corollaries, we first generalize Lemma~\ref{lem:betajsmall}, Lemma~ \ref{lem:typewlog} and a partial result from Lemma~\ref{lem:lbfory} to more general inequalities of the shape $|F_t(X,Y)|\leq Q$.

\begin{customlem}{\ref{lem:betajsmall}*}\label{lem:betajsmall-star}
Let $|t|\in \qi$, $|t|\geq 100$ and $Q \in \R^+$. Let $(x,y)\in \qi^2$ with $y\neq 0$ be a solution
to the inequality 
\[
	|F_t(X,Y)|\leq Q.
\]
If $(x,y)$ is of type $j$
(i.e.\ $|\beta^{(j)}|$ is minimal),
then 
\begin{equation*}
	|x - \alpha^{(j)} y|
	=|\beta^{(j)}|
	<\frac{8.86 \cdot Q}{|t|\cdot |y|^3}.
\end{equation*}
\end{customlem}
\begin{proof}
The proof is completely analogous to the proof of Lemma~\ref{lem:betajsmall}.
\end{proof}

\begin{customlem}{\ref{lem:typewlog}*}\label{lem:typewlog-star}
Let $t \in \qi$, $|t|\geq 100$ and $Q \in \R^+$. 
Let $(x,y)\in \qi^2$ be a solution to the inequality
\begin{equation}\label{ineq:Q}
	|F_t(X,Y)|\leq Q,
\end{equation}
with $\min\{|x|,|y|\} \geq \left( \frac{20.14 Q}{|t|} \right)^{1/4}$. Then $(x,y)$ is of type $j$ if and only if $(-y,x)$ is a solution of type $j+2\pmod{4}$.
\end{customlem}
\begin{proof}
First, let us define approximations to the roots of $f_t(X)$:
\[
	\xi^{(0)} = 0, \quad
	\xi^{(1)} = -1, \quad
	\xi^{(2)} = t, \quad
	\xi^{(3)} = 1.
\]
Then by Lemma~\ref{lem:roots-approx} we have that $|\alpha^{(i)} - \xi^{(i)}| < 0.06$ for $i=1,2,3,4$.

Let $(x,y)\in \qi^2$ be any solution to \eqref{ineq:Q} with $xy\neq 0$ of any type $j$. 
Then by Lemma~\ref{lem:betajsmall-star} we have that
\[
	\left|\alpha^{(j)} - \frac{x}{y}\right| 
	< \frac{8.86 \cdot Q}{|t|\cdot |y|^4}.
\]
Now note that the assumption 
\[
	|y|	
	\geq \min\{|x|,|y|\}
	> \left( \frac{20.14 Q}{|t|} \right)^{1/4}
	> \left( \frac{8.86 Q}{0.42 |t|} \right)^{1/4}
\]
was chosen such that\[
	\left|\alpha^{(j)} - \frac{x}{y}\right| 
	< \frac{8.86 \cdot Q}{|t|\cdot |y|^4}
	< 0.44
\]
In particular, this implies 
\begin{equation}\label{ineq:xi}
	\left|\xi^{(j)} - \frac{x}{y}\right|
	 < 0.5.
\end{equation}
Since the distance between any two distinct $\xi^{(i)}$ is at least 0.5, we can say that $(x,y)$ is of type $j$ if and only if it satisfies \eqref{ineq:xi}.

Now consider $(-y,x)$, which is also a solution to \eqref{ineq:Q} therefore of some type $k$, i.e.\
\begin{equation*}
	\left|\xi^{(k)} - \frac{-y}{x}\right|
	 < 0.5.
\end{equation*}
Thus we have a complex number $z = x/y$ which satisfies both $|z - \xi^{(j)}| < 0.5$ and $|-z^{-1} - \xi^{(k)}| < 0.5$. Looking at the set $\{\xi^{(0)} = 0, \xi^{(1)} = -1, \xi^{(2)} = t,  \xi^{(3)} = 1\}$ it is easy to see that this is only possible if either $\{j,k\} = {0,2}$ or $\{j,k\} = {1,3}$.
\end{proof}

\begin{rem}
The argument in the proof of Lemma~\ref{lem:typewlog-star} could also have been used to prove Lemma~\ref{lem:typewlog} without the assumption  $\min\{|x|,|y|\} \geq 3$. However, the lower bound 3 was also helpful for establishing the lower bounds for $|y|$ in Section~\ref{sec:lowerbound}. In particular, we used $|y|\geq 3$ to get a sufficiently large bound of the shape $|y| > c \cdot |t|$ in Step~\ref{step:type3-1} in Section~\ref{sec:lowerbound-type3}. Moreover, it was interesting to see some small solutions in Section~\ref{sec:smallsols}.
\end{rem}

Finally, as a last preparation for the proofs of the Corollaries, we prove the following lemma. It is analogous to a partial result in the proof the of Lemma~\ref{lem:lbfory}, which gave us lower bounds for $|y|$. Indeed, we will later use Lemma~\ref{lem:lbfory-star} to obtain lower bounds for $|y|$.

\begin{customlem}{\ref{lem:lbfory}*}\label{lem:lbfory-star}
Let, $t \in \qi$, $|t|\geq 100$ and $Q \in R^+$. 
Let $(x,y)\in \qi^2$ with $x \neq y\neq 0$ be a solution to the inequality 
\[
	|F_t(X,Y)|\leq Q.
\]
Assume that $(x,y)$ is either of type 0 or 3.
Then we have
\begin{equation*}\label{eq:lby-cors}
	1
	< \frac{2.16}{|t|}|y| + \frac{8.86 Q}{|t|\cdot |y|^3}.
\end{equation*}
\end{customlem}
\begin{proof}
We combine Lemma~\ref{lem:roots-approx} and Lemma~\ref{lem:betajsmall-star} in the same way as in Step~\ref{step:type0-1} in Section~\ref{sec:lowerbound-type0} and in Step~\ref{step:type3-1} in Section~\ref{sec:lowerbound-type3}.
In the type 0 case we obtain
\begin{align*}
	1	
	\leq |x|
	< \frac{1.01}{|t|}|y| + \frac{8.86  Q}{|t|\cdot |y|^3}
%	\leq \left(1.01 + \frac{8.86 Q}{|y|^4}  \right) \frac{|y|}{|t|}.
\end{align*}
and in the type 3 case we obtain 
\begin{align*}
	1	
	\leq |x - y|
	< \frac{2.16}{|t|}|y| + \frac{8.86 Q}{|t|\cdot |y|^3}.
%	\leq \left(2.16 + \frac{8.86 \cdot Q}{|y|^4}  \right) \frac{|y|}{|t|}.
\end{align*}
Overall, we have proven the Lemma.
\end{proof}

\begin{proof}[Proof of Corollary \ref{cor:lin}] 
Let $C>0$ be given.

First, we choose a constant $t_0$ such that $\kappa(t_0)<2$. This works for any $t_0\geq 524$, however, if $t_0$ is close to 524, we will have to choose $C_0$ extremely large. In fact, we choose $C_0$ in the following way:
\[
	C_0 = \max \{
		(20.14 C)^{1/4},
		3 C^{1/3},
		(443C)^{1/(2-\kappa(t_0))}
	\}.
\]
The motivation for this choice will become apparent later in the proof. 

To prove Corollary~\ref{cor:lin}, we need to check the following statement: 
For any square-free integer $d$ and any $t\in \qi$ with $|t|\geq t_0$, the inequality 
\begin{equation}\label{ineq:cor-lin-inproof}
	|F_t(X,Y)|\leq C |t|
	\quad \text{in } X,Y \in \qi
\end{equation}
has no solutions $(x,y)$ with $\min\{|x|,|y|\}\geq C_0$, except solutions of the shape $(x,\pm x)$ with $|x| \leq (C|t|/4)^{1/4}$.

Let $(x,y)$ be a solution to \eqref{ineq:cor-lin-inproof} with 
$\min\{|x|,|y|\}\geq C_0$, for some $t$ with $|t|\geq t_0$.

First, assume that $y = \pm x$. Then we get that $|F_t(x,y)| = 4|x|^4 \leq C|t|$, which implies  $|x| \leq (C|t|/4)^{1/4}$.

From now on, assume that $y \neq \pm x$ (this will be necessary for the application of Lemma~\ref{lem:lbfory-star} later).
In order to use Lemma~\ref{lem:typewlog-star}, we need to check that  
\[
	\min\{|x|,|y|\} 
	\geq \left( \frac{20.14 Q}{|t|} \right)^{1/4}
	= (20.14 C)^{1/4}.
\]
This is indeed guaranteed by $\min\{|x|,|y|\} \geq C_0 \geq (20.14 C)^{1/4}$. Thus by Lemma~\ref{lem:typewlog-star} we may assume without loss of generality that $(x,y)$ is either of type 0 or of type 3.

Next, we can use Lemma~\ref{lem:lbfory-star}, which (with $Q = C|t|$) gives us
\begin{align*}
 	1
 	< \frac{2.16}{|t|} |y| + \frac{8.86 C}{|y|^3}.
\end{align*}
Since we are assuming $|y| \geq C_0 \geq  3C ^{1/3}$,
we obtain
\begin{equation}\label{ineq:cor1-beforelb}
 	1 < \frac{2.16}{|t|} |y| + 0.33,
\end{equation}
which implies
\begin{equation}\label{eq:lb-cor1}
	|y| > 0.31 |t|.
\end{equation}
In particular, we have $|y|\geq 0.28 |t|$, and we can apply 
Proposition~\ref{prop:approx} and combine it with Lemma~\ref{lem:betajsmall-star} (with $Q = C|t|$):
\begin{align*}
	\frac{1}{15.48 |t| |y|^{\kappa + 1}}
	< \left| \alpha^{(j)} - \frac{x}{y} \right|
	< \frac{8.86 \cdot C \cdot |t|}{|t|\cdot |y|^4}.
\end{align*}
This implies
\begin{align*}
	|y|^{3-\kappa} 
	< C_2 |t|,
\end{align*}
where $C_2 = 8.86 \cdot C \cdot 15.48$.
Combining this with \eqref{eq:lb-cor1} we obtain
\[
	|y|^{3-\kappa} < C_3 |y|,
\]
with $C_3 = 443 C > C_2 / 0.31$. 
Now the inequality $|y|^{3-\kappa} <443 C |y|$ implies $|y|^{2-\kappa} <443 C$ and thus 
\begin{equation}\label{ineq:finish-cor1}
	|y| < (443 C)^{1/(2 - \kappa)}.
\end{equation}
Assume for a moment that $443 C<1$. Then $|y|<1$, which implies $y=0$ and is excluded because of $C_0>0$. Thus we may assume that $443 C \geq 1$.
Since $\kappa = \kappa(t)\leq \kappa(t_0)$ for all $t\geq t_0$, inequality~\eqref{ineq:finish-cor1} then implies $|y| < (443 C)^{1/(2 - \kappa(t_0))}$.
This contradicts our assumption $|y|\geq C_0 \geq (443 C)^{1/(2 - \kappa(t_0))}$.
\end{proof}

\begin{proof}[Proof of Corollary \ref{cor:eps}]
Let $0 < \eps < 1$. We need prove that there exists an effectively computable constant $t_0\geq 100$ such that the following statement holds: 
For any square-free integer $d$ and any $t\in \qi$ with $|t|\geq t_0$ the inequality 
\begin{equation}\label{ineq:cor-eps-inproof}
	|F_t(X,Y)|\leq |t|^{2-\eps}
	\quad \text{in } X,Y \in \qi
\end{equation}
has no solutions $(x,y)$ with  $\min\{|x|,|y|\} > (|t|^{2-\eps}/4)^{1/4}$.

Assume that there exists a solution $(x,y)$ to \eqref{ineq:cor-eps-inproof} with 
$\min\{|x|,|y|\} > (|t|^{2-\eps}/4)^{1/4}$.

If $y = \pm x$. Then we get that $|F_t(x,y)| = 4|x|^4 \leq |t|^{2-\eps}$, which implies  $|x| \leq (|t|^{2-\eps}/4)^{1/4}$, a contradiction.
From now on, assume that $y \neq \pm x$ (this will be necessary for the application of Lemma~\ref{lem:lbfory-star} later).

In order to use Lemma~\ref{lem:typewlog-star}, we need to check that  
\[
	\min\{|x|,|y|\} 
	\geq \left( \frac{20.14 Q}{|t|} \right)^{1/4}
	= (20.14 |t|^{1-\eps})^{1/4}
	= 20.14^{(1 - \eps)/4} \cdot |t|^{1/4 - \eps/4}.
\]
This is indeed guaranteed by $\min\{|x|,|y|\} > |t|^{1/2 - \eps/4}$, if $|t|$ is large enough.
Thus by Lemma~\ref{lem:typewlog-star} we may assume without loss of generality that $(x,y)$ is either of type 0 or of type~3.

Next, we can use Lemma~\ref{lem:lbfory-star} with $Q = |t|^{2-\eps}$ and we get
\begin{align*}
 	1
 	\leq \frac{2.16 |y|}{|t|} + \frac{8.86 |t|^{1-\eps}}{|y|^3}.
\end{align*}
Using $|y| > |t|^{1/2 - \eps/4}$, we obtain
\[
	1 
	< \frac{2.16 |y|}{|t|} + \frac{8.86 |t|^{1-\eps}}{ |t|^{3/2 - 3\eps/4}} 
	= \frac{2.16 |y|}{|t|} + \frac{8.86}{|t|^{1/2 + \eps/4}}.
\]
If $|t|$ is large enough, the last summand is at most $0.33$. Thus, as in the previous proof, we end up with inequality~\eqref{ineq:cor1-beforelb}, which implies
\begin{equation}\label{eq:lb-cor2}
	|y| > 0.31 |t|.
\end{equation}
In particular, we have $|y|\geq 0.28 |t|$, and we can apply 
Proposition~\ref{prop:approx} and combine it with Lemma~\ref{lem:betajsmall-star}:
\begin{align*}
	\frac{1}{15.48 |t| |y|^{\kappa + 1}}
	< \left| \alpha^{(j)} - \frac{x}{y} \right|
	< \frac{8.86 |t|^{1-\eps}}{|y|^4}.
\end{align*}
This implies
\begin{align*}
	|y|^{3-\kappa} 
	< 137.16 |t|^{2-\eps}.
\end{align*}
Combining that with \eqref{eq:lb-cor2}, we obtain
\[
	|y|^{3-\kappa} < 137.16 / 0.31^{2-\eps} \cdot |y|^{2-\eps},
\]
which implies
\[
	|y| 
	< (137.16 / 0.31^{2-\eps})^{1/(1+\eps-\kappa(t))}
	< (|t|^{2-\eps}/4)^{1/4},
\]
a contradiction. 

Note that the last inequality holds if $|t|$ is large enough, since $\kappa(t)$ gets arbitrarily close to 1 if $|t|$ is large enough and $(|t|^{2-\eps}/4)^{1/4}$ grows as $|t|$ grows.
In all the arguments of this proof ``large enough'' was always effectively computable in terms of $\eps$. Thus we have proven Corollary~\ref{cor:eps}.
\end{proof}

\section*{Acknowledgments}

B.F.\ thanks OWSD and Sida (Swedish International Development Cooperation Agency) for a scholarship during her Ph.D.\ studies at Wits, during which she started this project. We thank Paul Voutier, Volker Ziegler and Clemens Heuberger for useful conversations. Also thanks to CIRM, AIMS-Sengal, and A Room of One's Own for funding time and space to do research. 

I.V.\ was supported by the Austrian Science Fund (FWF) under the project I4406, as well as by the Austrian Marshall Plan Foundation with a Marshall Plan Scholarship. She wants to thank Franklin \& Marshall College for their very friendly and generous hospitality.

\bibliographystyle{habbrv}
\bibliography{refs}

\begin{thebibliography}{10}
\expandafter\ifx\csname url\endcsname\relax
  \def\url#1{\texttt{#1}}\fi
\expandafter\ifx\csname doi\endcsname\relax
  \def\doi#1{\burlalt{doi:#1}{http://dx.doi.org/#1}}\fi
\expandafter\ifx\csname urlprefix\endcsname\relax\def\urlprefix{URL: }\fi
\expandafter\ifx\csname href\endcsname\relax
  \def\href#1#2{#2}\fi
\expandafter\ifx\csname burlalt\endcsname\relax
  \def\burlalt#1#2{\href{#2}{#1}}\fi

\bibitem{ChenVoutier1997}
J.~H. Chen and P.~Voutier.
\newblock Complete solution of the {D}iophantine equation {$X^2+1=dY^4$} and a
  related family of quartic {T}hue equations.
\newblock {\em J. Number Theory}, 62(1):71--99, 1997.
\newblock \doi{10.1006/jnth.1997.2018}.

\bibitem{ChenVoutier1997_2018}
J.~H. Chen and P.~Voutier.
\newblock Complete solution of the {D}iophantine equation {$X^2+1=dY^4$} and a
  related family of quartic {T}hue equations, 2018.
\newblock \href{https://arxiv.org/abs/1401.5450v2}{arXiv:1401.5450v2}.

\bibitem{Chudnovsky1983}
G.~V. Chudnovsky.
\newblock On the method of {T}hue-{S}iegel.
\newblock {\em Ann. of Math. (2)}, 117(2):325--382, 1983.
\newblock \doi{10.2307/2007080}.

\bibitem{FayeVukusicWaxmanZiegler2023}
B.~Faye, I.~Vukusic, E.~Waxman, and V.~Ziegler.
\newblock Thue equations over $\mathbb{C}(t)$: The complete solution of a
  simple quartic family, 2023.
\newblock \href{https://arxiv.org/abs/2301.06129}{arXiv:2301.06129}.

\bibitem{FuchsZiegler2006}
C.~Fuchs and V.~Ziegler.
\newblock Thomas's family of {T}hue equations over function fields.
\newblock {\em Q. J. Math.}, 57(1):81--91, 2006.
\newblock \doi{10.1093/qmath/hah062}.

\bibitem{GaalJadrijevicRemete2018}
I.~Ga\'{a}l, B.~Jadrijevi\'{c}, and L.~Remete.
\newblock Totally real {T}hue inequalities over imaginary quadratic fields.
\newblock {\em Glas. Mat. Ser. III}, 53(73)(2):229--238, 2018.
\newblock \doi{10.3336/gm.53.2.02}.

\bibitem{GaalJadrijevicRemete2019}
I.~Ga\'{a}l, B.~Jadrijevi\'{c}, and L.~Remete.
\newblock Simplest quartic and simplest sextic {T}hue equations over imaginary
  quadratic fields.
\newblock {\em Int. J. Number Theory}, 15(1):11--27, 2019.
\newblock \doi{10.1142/S1793042118501695}.

\bibitem{GaalPohst2002}
I.~Ga\'{a}l and M.~Pohst.
\newblock On the resolution of relative {T}hue equations.
\newblock {\em Math. Comp.}, 71(237):429--440, 2002.
\newblock \doi{10.1090/S0025-5718-01-01329-1}.

\bibitem{Heuberger2006}
C.~Heuberger.
\newblock All solutions to {T}homas' family of {T}hue equations over imaginary
  quadratic number fields.
\newblock {\em J. Symbolic Comput.}, 41(9):980--998, 2006.
\newblock \doi{10.1016/j.jsc.2006.05.001}.

\bibitem{Heuberger2006-surv}
C.~Heuberger.
\newblock Parametrized {T}hue {E}quations -- {A} {S}urvey.
\newblock In {\em Proceedings of the RIMS symposium “Analytic Number Theory
  and Surrounding Areas”}, volume 1511 of {\em RIMS K{\^o}ky{\^u}roku}, pages
  82--91, 2006.
\newblock
  \urlprefix\url{https://www.kurims.kyoto-u.ac.jp/~kyodo/kokyuroku/contents/pdf/1511-11.pdf}.

\bibitem{HeubergerPethoTichy2002}
C.~Heuberger, A.~Peth\H{o}, and R.~F. Tichy.
\newblock Thomas' family of {T}hue equations over imaginary quadratic fields.
\newblock {\em J. Symbolic Comput.}, 34(5):437--449, 2002.
\newblock \doi{10.1006/jsco.2002.0568}.

\bibitem{JadrijevicZiegler2006}
B.~Jadrijevi\'{c} and V.~Ziegler.
\newblock A system of relative {P}ellian equations and a related family of
  relative {T}hue equations.
\newblock {\em Int. J. Number Theory}, 2(4):569--590, 2006.
\newblock \doi{10.1142/S1793042106000735}.

\bibitem{KirschenhoferLamplThuswaldner2007}
P.~Kirschenhofer, C.~M. Lampl, and J.~M. Thuswaldner.
\newblock On a parameterized family of relative {T}hue equations.
\newblock {\em Publ. Math. Debrecen}, 71(1-2):101--139, 2007.

\bibitem{LettlPethoVoutier1999}
G.~Lettl, A.~Peth\H{o}, and P.~Voutier.
\newblock Simple families of {T}hue inequalities.
\newblock {\em Trans. Amer. Math. Soc.}, 351(5):1871--1894, 1999.
\newblock \doi{10.1090/S0002-9947-99-02244-8}.

\bibitem{Mignotte1993}
M.~Mignotte.
\newblock Verification of a conjecture of {E}. {Thomas}.
\newblock {\em J. Number Theory}, 44(2):172--177, 1993.
\newblock \doi{10.1006/jnth.1993.1043}.

\bibitem{sagemath}
{The Sage Developers}.
\newblock {\em {S}ageMath, the {S}age {M}athematics {S}oftware {S}ystem
  ({V}ersion 9.2)}, 2021.
\newblock \urlprefix\url{https://www.sagemath.org}.

\bibitem{Thomas1990}
E.~Thomas.
\newblock Complete solutions to a family of cubic {D}iophantine equations.
\newblock {\em J. Number Theory}, 34(2):235--250, 1990.
\newblock \doi{10.1016/0022-314X(90)90154-J}.

\bibitem{Thue1909}
A.~Thue.
\newblock \"{U}ber {A}nn\"{a}herungswerte algebraischer {Z}ahlen.
\newblock {\em J. Reine Angew. Math.}, 135:284--305, 1909.
\newblock \doi{10.1515/crll.1909.135.284}.

\bibitem{Thue1910}
A.~Thue.
\newblock Ein {F}undamentaltheorem zur {B}estimmung von {A}nn\"{a}herungswerten
  aller {W}urzeln gewisser ganzer {F}unktionen.
\newblock {\em J. Reine Angew. Math.}, 138:96--108, 1910.
\newblock \doi{10.1515/crll.1910.138.96}.

\bibitem{TzanakisDeWeger1989}
N.~Tzanakis and B.~M.~M. de~Weger.
\newblock On the practical solution of the {T}hue equation.
\newblock {\em J. Number Theory}, 31(2):99--132, 1989.
\newblock \doi{10.1016/0022-314X(89)90014-0}.

\bibitem{Voutier2010}
P.~M. Voutier.
\newblock Thue's {F}undamentaltheorem, {I}: {T}he general case.
\newblock {\em Acta Arith.}, 143(2):101--144, 2010.
\newblock \doi{10.4064/aa143-2-1}.

\bibitem{Wakabayashi2007}
I.~Wakabayashi.
\newblock Simple families of {T}hue inequalities.
\newblock {\em Ann. Sci. Math. Qu\'{e}bec}, 31(2):211--232 (2008), 2007.

\bibitem{Ziegler2005}
V.~Ziegler.
\newblock On a family of cubics over imaginary quadratic fields.
\newblock {\em Period. Math. Hungar.}, 51(2):109--130, 2005.
\newblock \doi{10.1007/s10998-005-0032-6}.

\bibitem{Ziegler2006}
V.~Ziegler.
\newblock On a family of relative quartic {T}hue inequalities.
\newblock {\em J. Number Theory}, 120(2):303--325, 2006.
\newblock \doi{10.1016/j.jnt.2005.12.004}.

\end{thebibliography}

\section*{Appendix}\setcurrentname{Appendix}\label{sec:appendix}

Let $m>0$ be some given bound.
We describe how to give a list of all imaginary quadratic integers $0<|x| \leq m$ with either $\Im(x)>0$ or $0<x\in \Z$ (i.e.\ up to sign the full list of quadratic integers in the given range).

Any quadratic integer $x \in \qi$ can be written as
\[
	x=a+b\omega
	\quad \text{with} \quad
	a,b\in \Z
	\quad \text{and} \quad
	\omega = \begin{cases}
		\frac{1+\sqrt{-d}}{2}, &\text{if} \quad -d\equiv 1 \pmod{4},\\
		\sqrt{-d},  &\text{else}.
	\end{cases}
\]
Then we have
\begin{equation}\label{eq:smally_absy}
	|x|^2 = \begin{cases}
		{(a+\frac{b}{2})^2 + d(\frac{b}{2})^2}, &\text{if} \quad -d\equiv 1 \pmod{4},\\
		{a^2 + d b^2},  &\text{else}.
	\end{cases}	
\end{equation}
Therefore, we only need to check $d$'s with the following properties: $d\geq 1$, $d$ is square free and $1/4+d/4 \leq m^2$ if $d\equiv 1 \pmod{4}$ and $d \leq m^2$ else. For example, for $m=3$, we get that
\[
	d \in \{ 1,2,3,5,6,7,11,15,19,23,31,35\}.
\]
Then for each $d$ we need to find all $a,b\in \Z$ such that $|a+b\omega|\leq m$.
If we first add all integers $1,2 \bb \floor{m}$ to our list, we can assume that $b\geq 1$.
 From~\eqref{eq:smally_absy} we get the following bounds: If $-d\equiv 1 \pmod{4}$, then
\[
	1 \leq b \leq 2m/\sqrt{d} 
	\quad \text{and} \quad
	- \sqrt{m^2-d\left(\frac{b}{2}\right)^2} - \frac{b}{2}		
	\leq a 
	\leq \sqrt{m^2-d\left(\frac{b}{2}\right)^2} - \frac{b}{2}.
\]
If $-d\equiv 2,3 \pmod{4}$, then
\[
	1 \leq b \leq m/\sqrt{d} 
	\quad \text{and} \quad
	- \sqrt{m^2-db^2} \leq a \leq \sqrt{m^2-db^2}.
\]
%\begin{align*}
%\text{If} \quad -d\equiv 1 \pmod{4} \quad \text{then,} \quad
%	&|b|\leq 2m/\sqrt{d} 
%	\quad \text{and} \quad
%	|a|< \sqrt{m^2-d\left(\frac{b}{2}\right)^2}+\frac{|b|}{2}. \\
%\text{Else,} \quad 
%	&|b|< m/\sqrt{d} 
%	\quad \text{and} \quad
%	|a|< \sqrt{m^2-db^2}.
%\end{align*}
Then we only need to loop through all such $d$'s, $b$'s and $a$'s.

For example, for $m=3$ we obtain a list of $76$ quadratic integers; here is a very short summary:
\[
	x \in \{
	1, 2, 3, 
	\pm 2 + i, \pm 1 + i, i, 
	2i \pm 2
	\bb
	\frac{\pm 1+\sqrt{-35}}{2}
	\}.
\]

Finally, in order to not rely on hyperlinks, here are the links to the Sage code, that has been referred to in the paper: 

\url{https://cocalc.com/IngridVukusic/QuarticFamily/Irreducibility} for the proof of Lemma~\ref{lem:irred},

\url{https://cocalc.com/IngridVukusic/QuarticFamily/SmallSolutions} for the proof of Lemma~\ref{lem:smallSols}, 

\url{https://cocalc.com/IngridVukusic/QuarticFamily/Rouche} for the proofs of Lemma~\ref{lem:roots-approx}, \ref{lem:betterApprox} and \ref{lem:betterApproxalpha3}, 

\url{https://cocalc.com/IngridVukusic/QuarticFamily/LowerBoundForY} for the proof of Lemma~\ref{lem:lbfory}.

\end{document}